\documentclass[11pt]{amsart}
\usepackage{amssymb,latexsym,amsmath}
\usepackage{hyperref}
\input{xypic}
\xyoption{all}

\newcommand{\E}{E}

\newcommand{\Q}{{\mathbb Q}}
\newcommand{\bQ}{\overline{{\mathbb Q}}}
\newcommand{\R}{{\mathbb R}}
\newcommand{\Z}{{\mathbb Z}}
\newcommand{\C}{{\mathbb C}}

\newcommand{\A}{{\mathbb A}}

\newcommand{\G}{{\mathbb G}}
\newcommand{\M}{{\mathbb M}}

\newcommand{\fr}{{\mathfrak r}}
\newcommand{\fs}{{\mathfrak s}}

\newcommand{\fS}{\mathsf{s}}
\newcommand{\fT}{\mathsf{t}}
\newcommand{\fC}{\mathsf{c}}

\newcommand{\cF}{\mathcal{F}}

\newcommand{\cT}{\mathcal{T}}

\newcommand{\sF}{\mathsf{F}}

\newcommand{\GL}{{\rm GL}}

\newcommand{\OO}{{\rm O}}

\newcommand{\GO}{{\rm GO}}
\newcommand{\GSO}{{\rm GSO}}

\newcommand{\GSp}{{\rm GSp}}

\newcommand{\Sym}{{\rm Sym}}
\newcommand{\As}{{\rm As}}

\newcommand{\Hom}{{\rm Hom}}
\newcommand{\End}{{\rm End}}

\newcommand{\Gal}{{\rm Gal}}

\newcommand{\Res}{{\rm Res}}
\newcommand{\Spec}{{\rm Spec}}
\newcommand{\Aut}{{\rm Aut}}

\theoremstyle{plain}
\newtheorem{thm}[subsubsection]{Theorem}
\newtheorem{lem}[subsubsection]{Lemma}
\newtheorem{cor}[subsubsection]{Corollary}
\newtheorem{prop}[subsubsection]{Proposition}

\theoremstyle{definition}
\newtheorem{remark}[subsubsection]{Remark}
\newtheorem{Def}[subsubsection]{Definition}

\newtheorem{ex}[subsubsection]{Example}

\newtheorem{lem*}{Lemma}[subsubsection]
\newtheorem{prop*}{Proposition}[subsubsection]

\begin{document}

\title{\bf Motives, Periods, and Functoriality}

\author{\bf Pierre Deligne \ \ \& \ \ A. Raghuram}
\address{Institute for Advanced Study, 1 Einstein Drive, Princeton, NJ 08540, USA.}
\email{deligne@ias.edu}

\address{Dept.\,of Mathematics, Fordham University at Lincoln Center, New York, NY 10023, USA.}
\email{araghuram@fordham.edu}

\thanks{}
\date{\today}

\begin{abstract}
Given a pure motive $M$ over $\Q$ with a multilinear algebraic structure $\fS$ on $M$, and given a representation $V$ of the group respecting $\fS$, we describe 
a functorial transfer $M^V$. We formulate a criterion that guarantees when the two periods of $M^V$ are equal. This has an implication for the critical values of the $L$-function attached to $M^V.$ The criterion is explicated in a variety of examples such as: tensor product motives and Rankin--Selberg $L$-functions; orthogonal motives and the standard $L$-function for even orthogonal groups; twisted tensor motives and Asai $L$-functions. 
\end{abstract}

\maketitle

{\tiny  \tableofcontents}

\section{Introduction}
\label{sec:intro}

The second author, in independent collaborations with G\"unter Harder, Chandrasheel Bhagwat, and Muthu Krishnamurthy, has found a number of cases where 
for some reductive group $G$ over a number field $F$, some cuspidal automorphic representation $\pi$ of $G/F$, and some finite-dimensional representations $\rho$ of the Langlands dual ${}^LG$ of $G$, the 
corresponding $L$-function $L(s, \pi, \rho) =: L(s)$ is such that for some values of $s$, the ratio $L(s)/L(s+1)$ is algebraic: 
\begin{itemize}
\item Rankin--Selberg $L$-functions for $\GL_n \times \GL_m$ over a totally real field \cite{harder-raghuram-book}; 
\item Rankin--Selberg $L$-functions for $\GL_n \times \GL_m$ over a totally imaginary field \cite{raghuram-imaginary}, or $L$-functions for 
$\GL_n$ over a CM field \cite{raghuram-CM}; 
\item Degree-$2n$ $L$-functions for $\GL_1 \times \OO(2n)$ over a totally real field \cite{bhagwat-raghuram};
\item Degree-$n^2$ Asai $L$-functions over $F$ attached to cuspidal representations of $\GL_n/E$; with $E/F$ a quadratic extension of totally real fields \cite{muthu-raghuram}. 
\end{itemize} 
In all cases, building on previous work of Harder \cite{harder-inventiones}, \cite{harder-tifr}, 
rationality results follow from a study of rank-one Eisenstein cohomology on an ambient reductive group having $G$ as a Levi quotient, giving a cohomological interpretation to some aspects of the Langlands--Shahidi method \cite{shahidi-book}.  

\smallskip

In all cases, there is conjecturally a motive $M$ such that, possibly after a shift $s' = s+a$, $L(s)$ is the $L$-function attached to $M$:
$$
L(s) \ = \ L(M, s')
$$
and the values $s$ and $s+1$ considered for $L(s)$ correspond to values $s'$ and $s'+1$ which are critical for $M$. 

\smallskip

In \cite{deligne}, the first author proposed a conjecture on the critical values $L(M, s)$ of motivic $L$-functions. The $L$-functions considered were without the $L$-factors at infinity. 
For a motive $M$ with coefficients in a number field $E$, the $L$-function takes values in $E \otimes \C$, and for $s$ critical for $M$, the conjecture determined 
$L(M, s)$ up to multiplication by an element of $E$. It was given by some determinant of periods $c^+(M(s))$ of the motive $M(s)$. Here, $s$ is always an integer, and 
$M(s)$ is a Tate-twist of $M$. There is another determinant of periods $c^-(M(s))$, such that $c^+(M(s+1)) = c^-(M(s))$ up to multiplication by some power of $2\pi i$. Reintroducing 
the $L$-factors at infinity also introduces powers of $\pi$ and $i.$   

\smallskip
The aim of the present work is to prove rationality results for some ratios $c^+(M)/c^-(M)$ and to explain how they {\it a posteriori} predict some of the results 
of the first author mentioned above. The interest is two-fold. In one direction, it gives partial confirmations of \cite[Conj.\,2.8]{deligne}. In the other direction, it has led 
to corrections to the results on special values of automorphic $L$-functions in the literature (see the commentary in \cite[Sect.\,6]{raghuram-pias}), 
and in the results of the first author (especially, see \cite[Sect.\,5.4]{raghuram-imaginary} and \ref{sec:tot-imag-field-comparison} below). It 
also suggested new cases to be investigated. 

\smallskip
In Sect.\,\ref{sec:periods-motives}, after reviewing the basics of a critical motive and its periods, we describe the realizations of a motive $M^V$ attached to an algebraic 
representation $V$ of the group of automorphisms of a realization of a motive $M$ that respects some given multilinear algebraic structure $\fS$ on $M.$ The assignment 
$M \mapsto M^V$ is to be construed as the motivic analogue of Langlands transfer.  
We then state and prove the main result on period relations; see Thm.\,\ref{thm:criterion-M-V}; the proof is inspired by the methods of Yoshida \cite{yoshida}.
In Sect.\,\ref{sec:examples} we discuss a variety of concrete examples relevant to the above mentioned automorphic $L$-functions; especially, see 
Prop.\,\ref{prop:tensor-product-even-even}, Thm.\,\ref{thm:c+/c-over_total_imag}, Prop.\,\ref{prop:tensor_L_E}, 
Thm.\,\ref{thm:asai-motives}, and Thm.\,\ref{thm:gsp-go}.

\medskip
\section{The main theorem on critical motives}
\label{sec:periods-motives}

The $L$-function of a motive over a number field $F$ coincides with the $L$-function of the motive over $\Q$ deduced from it by restriction of scalars. Because of
this, we will often consider only (pure) motives over $\Q.$ They will be thought in terms of corresponding system of realizations in Betti, 
de~Rham, and $\ell$-adic cohomologies, and the comparison isomorphisms. We will also consider motives with coefficients in a number field $\E$, that 
is given with $\E \to \mathrm{End}(M).$

\smallskip

Let $M$ be a motive with coefficients in $\E$. We will review what it means for $M$ to be critical, and we will discuss the definition of the periods 
$c^\pm(M) \in (\E \otimes \C)^*/\E^*$ attached by the first author (\cite{deligne}) to critical motives. We will exhibit cases where $c^+(M) = c^-(M).$ The inspiration for the proof comes from Yoshida \cite{yoshida}. 

\smallskip

The periods $c^\pm(M)$ are interesting because of their conjectural relation with the critical values of the $L$-function attached to $M.$ The $\ell$-adic realizations are needed to define this $L$-function, but the definition of $c^\pm(M)$ depends only on the Betti and de~Rham realizations, and our criterion
for their equality can be stated and proved for systems of Betti and de~Rham realizations which are defined in Sect.\,\ref{sec:Betti--deRham} below.  

\smallskip

Motives admit a tensor product and duals, so that one can can consider a motive $M$ together with a multilinear algebraic structure $\fS$. From 
$(M, \fS)$, one can obtain new motives $M'$ by multilinear algebraic constructions. Here is a way to formalize this. Fix a realization functor $R$ with 
values in $\Q$-vector spaces, for instance Betti or de~Rham. Let $G_R \subset \GL(M_R)$ ($R = \mathrm{B}$ or $\mathrm{dR}$) be the algebraic group
of automorphisms of $(M_R, \fS_R)$. The comparison isomorphism $M_{\rm B} \otimes \C \to M_{\rm dR} \otimes \C$ will induce an isomorphism 
$(G_B)_\C \to (G_{\rm dR})_\C.$ Then to any representation $V$ of $G_R$ corresponds a motive $M^V$ deduced from $M$; this construction being functorial 
and compatible with tensor product, and $M_R$ corresponding to $M.$ The transfer $M \mapsto M^V$ is the motivic analogue of Langlands transfer on the automorphic side.

\smallskip

Our aim is to find structures $\fS$ and representations $V$ for which $c^+(M^V) = c^-(M^V).$

\bigskip

\subsection{Betti--de Rham systems of realizations}
\label{sec:Betti--deRham}

A weight $w$ Betti--de Rham system of realizations 
 $M$ over $\Q$ consists of the following data (i)-(iii), obeying (iv): 

\begin{enumerate}
\item[(i)] {\it Betti realization:} for $C$ an algebraic closure of $\R$, a $\Q$-vector space $M_B\{C\}$, depending functorially on $C$. If 
$\C := \R[i]/(i^2+1)$ is the standard algebraic closure of $\R$, this amounts to the data of $M_B := M_B\{\C\},$ and of the 
involution $\sF_\infty$ induced by the complex conjugation automorphism of $\C$. We call the involution $\sF_\infty$ the Frobenius at 
infinity. It gives a decomposition $M_B = M^+ \oplus M^-$ of $M_B$ into its $\pm 1$ eigenspaces. 

\smallskip
\item[(ii)] {\it de Rham realization:} a $\Q$-vector space $M_{\rm dR}$ together with a finite decreasing filtration $\cF$, {\it the Hodge filtration}. 

\smallskip
\item[(iii)] {\it comparison isomorphism:} an isomorphism 
$M_{\rm B}\{C\} \otimes C \to M_{\rm dR} \otimes C$ of $C$-vector spaces that is functorial in $C.$

\end{enumerate}
Let $\sigma_{\rm B}$ (resp., $\sigma_{\rm dR}$) be the antilinear involution of $M_{\rm B} \otimes \C$ (resp., $M_{\rm dR} \otimes \C$) which is the tensor product of the identity and of the complex conjugation on $\C$. The action by functoriality of 
the complex conjugation on $M_B \otimes \C$ is $\sF_\infty \cdot \sigma_{\rm B}$. The data in (iii) amounts to the data of an isomorphism 
$M_{\rm B} \otimes \C \to M_{\rm dR} \otimes \C$ which carries $\sF_\infty \cdot \sigma_{\rm B}$ into $\sigma_{\rm dR}.$ In other words, 
$M^+ \otimes \R \, \oplus \, M^- \otimes i\R$ maps isomorphically onto $M_{\rm dR} \otimes \R.$

\begin{enumerate}
\smallskip
\item[(iv)] One requires that if the complexified Hodge filtration is transported by the comparison isomorphism to a filtration of $M_B \otimes \C,$ still 
denoted $\cF$, it and its complex conjugate $\bar{\cF} = \sigma_{\rm B}(\cF)$ are $w$-opposite, meaning that $M_{\rm B} \otimes \C$ is the direct sum 
of the $M^{p,q} := \cF^p \cap \bar{\cF}^q$ for $p+q = w.$ One calls the decomposition $M_{\rm B} \otimes \C = {\oplus}_{p+q = w} M^{p,q}$ 
{\it the Hodge decomposition}. 
One has $\cF^a = \oplus_{p \geq a} M^{p,q},$ and $M^{p,q}$ and $M^{q,p}$ are complex conjugates. As $\cF^p$ is dR-rational, 
it is stable under $\sigma_{\rm B} \cdot \sF_\infty$, identified with $\sigma_{\rm dR}.$ It follows that $\sF_\infty$ maps $\cF$ into $\bar{\cF}$, hence 
$M^{p,q}$ to $M^{q,p}.$
\end{enumerate}

One defines morphisms, tensor products, and duals in an obvious way. The Hodge decomposition is functorial. It follows that any morphism is strictly compatible with the Hodge filtrations and that the Betti--de Rham systems of realizations of weight $w$ form an abelian category. The tensor product of systems of realizations of weight
$w$ and $w'$ is of weight $w+w'$, and 
the dual of $M$ of weight $w$ is of weight $-w.$

\medskip

\subsubsection*{\bf Examples}
The motivating example is $M = H^n(X)$, the cohomology in degree $n$ of a smooth projective variety $X$ over $\Q$ with $M_{\rm B}\{C\} := H^n(X(C), \Q)$, and $M_{\rm dR}$ the algebraic de Rham cohomology of $X$; the 
weight is $n.$ 

\smallskip

The Tate object $\Q(1)$, of weight $-2$ and Hodge type $(-1,-1),$ is defined by $\Q(1)_{\rm B}\{C\} = 2\pi i \Q \subset C$, and 
$\Q(1)_{\rm dR} = \Q$ purely of filtration $-1$. As $\Q(1)$ is of rank one, its $n^{\rm th}$ tensor power, denoted $\Q(n)$, is defined for $n \in \Z$. For $n=-1$, $\Q(-1)$ is 
the dual of $\Q(1)$.  It is also $H^1(\G_m),$ and the $2\pi i$ comes from 
$\oint \frac{dz}{z} = 2\pi i.$ The $n^{\rm th}$ Tate twist of $M$ is defined as 
$M(n) := M \otimes \Q(n).$

\medskip

\subsection{Critical Betti--de Rham realizations}
\begin{Def}
We say that $M$ is critical if the comparison isomorphism induces an isomorphism 
\begin{equation}
\label{eqn:critical}
M^+ \otimes \C \ \longrightarrow \ (M_{\rm dR}/\cF^0) \otimes \C.
\end{equation}
\end{Def}

The $L$-factor at infinity $L_\infty(M,s)$ attached to $M$ is defined as follows (\cite[5.2]{deligne}, \cite{Se}). 
One defines $\Gamma_\C(s):= 2\cdot(2\pi)^{-s}\Gamma(s), \ \Gamma_\R(s):= \pi^{-s/2}\ \Gamma (s/2), \ h^{p,q}:=\dim M^{p,q};$ for $M$ of even 
weight $w,$ $w=2p,$ for $\epsilon= \pm 1$ define $h^{p,p}(\epsilon)$ to be dimension of the $\epsilon\cdot(-1)^p$ eigenspace of $\sF_{\infty}$ acting on $M^{p,p};$ 
then 
\begin{equation*}
L_{\infty} (M, s):= \left(\prod_{p<q}\ \Gamma_\C (s-p)^{h^{p,q}} \right) \cdot \Gamma_\R(s-p)^{h^{p,p}(+1)} \cdot \Gamma_\R(s+1-p)^{h^{p,p}(-1)}, 
\end{equation*}
where the last two factors are possibly there only when $M$ is of even weight $w=2p.$ 
One has $L_{\infty}(M(n),s)=L_{\infty}(M,s+n)$. 
Criticality of $M$ is equivalent to finiteness of both the values $L_\infty(M,0)$ and $L_\infty(M^{\sf v}(1), 0)$.

\smallskip

\subsubsection{\bf $c^+(M)$ for a critical $M$}
If $M$ is critical, one defines $c^+(M)$ to be the determinant of the isomorphism 
(\ref{eqn:critical}) computed in terms of $\Q$-bases of 
$M^+$ and $M_{\rm dR}/\cF^0.$ It is well-defined modulo multiplication by
$\lambda \in \Q^*$. As $M^+$ maps to $M_{\rm dR} \otimes \R$, the isomorphism in (\ref{eqn:critical})
is the complexification of an isomorphism 
\begin{equation}
\label{eqn:critical-R}
M^+ \otimes \R \ \longrightarrow \ (M_{\rm dR}/\cF^0) \otimes \R.
\end{equation}
Hence, if $M$ is critical, then $c^+(M)$ is a nonzero real number. 
If $\Lambda \subset M^+$ is a lattice, then 
$\Lambda \backslash ((M_{\rm dR}/\cF^0) \otimes \R)$  is a torus whose Lie algebra has the rational structure 
$(M_{\rm dR}/\cF^0)$, and $c^+(M)$ is the volume of that torus for a rational volume element. If $M$ has weight $w < 0$, this torus modulo torsion, 
that is $M^+ \backslash ((M_{\rm dR} / \cF^0)\otimes \R)$ is the group of extensions of $\Q(0)$ by $M$ in the category of mixed Hodge structures with a real Frobenius. 

\smallskip

\subsubsection{\bf $c^+(M)$ for a critical $M$ with coefficients}
More generally, if $M$ is with coefficients in $\E$, that is given with $\E \to \mathrm{End}(M)$, then $M^+$ and $M_{\rm dR}/\cF^0$ are $\E$-vector spaces. If $M$ is critical then 
(\ref{eqn:critical}) is an isomorphism of $\E \otimes \C$-modules. One then defines $c^+_{\E}(M)$ to be the determinant of this isomorphism, computed using $\E$-bases 
of $M^+$ and $M_{\rm dR} / \cF^0$. This determinant is real, that is belongs to $(\E \otimes \R)^*\subset (\E \otimes \C)^*$ and is well-defined modulo $\E^*$. The $c^+(M)$ previously defined is related to $c^+_\E(M)$ by
\begin{equation*}
c^+(M) = N_{\E/\Q}\ c^+_\E(M).
\end{equation*}
When this will not lead to ambiguity, we will write simply $c^+(M)$ for $c^+_\E(M)$.

The $L$-factor at infinity is defined as before, using dimensions over $\E$.

\smallskip

\subsubsection{\bf Conjecture on critical $L$-values}
If $M$ comes from a motive $\M$, and if enough conjectures are verified so that the $L$-function $L(\M,s)$ is defined, \cite{deligne} conjectures that 
for a critical $\M,$ 
$\ L(\M,0)$ can vanish only for $\M$ of weight $-1$, and that if $L(\M,0)\neq 0$, one has 
\begin{equation}
\label{eqn:conj-Q}
L(\M,0) \ \approx \ c^+(M) \quad \pmod{\Q^*}.
\end{equation}
The $L$-function considered is the finite part of the global $L$-function, i.e., without its natural factors at infinity $L_\infty(\M,s)$. 

\smallskip

For example, $\Q(n)$ is critical for $n> 0$ and even, in which case $\cF^0 = 0$, $\Q(n)^+ = \Q(n)_{\rm B}$, 
or for $n < 0$ and odd, in which case $\Q(n)_{\rm dR}/\cF^0 = 0 = \Q(n)^+.$ One has $L(\Q(n),s) = \zeta(s+n)$, and 
for $\Q(n)$ the above conjecture is a consequence of a famous theorem of Euler. 

\smallskip

More generally, if $\M$ is with coefficients in $\E$, one gets an $L$-function $L_\sigma(\M,s)$ for each embedding 
$\sigma : \E \to \C$. 
Using the isomorphism $\E \otimes \C \xrightarrow{\sim} \C^{\Hom(\E,\C)}$, one defines $L (\M,s)$, with values in $\E \otimes \C$ by 
$L(\M,s) = \{L_{\sigma}(\M,s)\}_{\sigma : \E \hookrightarrow \C}$. The conjecture in \cite{deligne} is that if $M$ is critical and that $L(\M,s)\neq 0,$ one has
\begin{equation}
\label{eqn:conj-E}
L(\M,0) \ \approx \ c^+_E(\M) \quad \pmod{\E^*}.
\end{equation}

\medskip

\subsection{Other determinant of periods}

\smallskip

\subsubsection{\bf The invariants $c^\pm(M)$}
We will say that $M$ has vanishing middle Hodge number if either its weight $w$ is odd, or its weight $w$ is even
and $M^{w/2, w/2} = 0.$ If $M$ of weight $2n$ or $2n+1$ has vanishing middle Hodge number, then, 
as $\sF_\infty$ exchanges 
$M^{p,q}$ with $M^{q,p}$, and that $\cF^n$ contains exactly one of them, the maps 
$$
M^\pm \otimes \C \ \to \ (M_{\rm dR}/\cF^n) \otimes \C
$$
are isomorphisms. One defines $c^\pm(M)$ to be their determinants computed using rational bases. As $M^+$ (resp., $M^-$) maps to 
$M_{\rm dR} \otimes \R$ (resp., $M_{\rm dR} \otimes i\R$) we get 
$$
c^+(M) \in \R \quad \mbox{(resp., $c^-(M) \in i^{\dim(M_{\rm B})/2} \R$)}. 
$$
If $w$ is an even integer, say $w = 2n$, and $\sF_\infty$ acts as the identity (resp., by $-1$) on $M^{n,n}$, 
the maps 
$$
M^+ \otimes \C \ \to \ (M_{\rm dR}/\cF^{n+1}) \otimes \C, \quad 
\mbox{(resp., $M^+ \otimes \C \ \to (M_{\rm dR}/\cF^{n}) \otimes \C$),}
$$
and 
$$
M^- \otimes \C \ \to \ (M_{\rm dR}/\cF^{n}) \otimes \C, \quad 
\mbox{(resp., $M^- \otimes \C \ \to \ (M_{\rm dR}/\cF^{n+1}) \otimes \C$)} 
$$
are isomorphisms. Again, one defines $c^\pm(M)$ to be their determinants computed using rational bases, and we get 
$$
c^+(M) \in \R \quad \mbox{(resp., $c^-(M) \in i^{\dim(M^-)}\R$)}. 
$$

In the case where $M$ has vanishing middle Hodge number, the comparison between $c^+(M)$ and $c^-(M)$ is made easier by the fact that 
the maps used in their definition have the same target space.

\smallskip

If $M$ is with coefficients in $\E$, one similarly defines invariants  
$c^\pm_E(M) \in (\E \otimes \C)^*$ which are well-defined modulo $\E^*.$

\smallskip

For $M$ to be critical of weight $w$, $\sF_\infty$ must, when $w$ is even, act by $+1$ or by $-1$ on $M^{w/2, w/2},$ and $\cF^0$ 
must coincide with $\cF^n$ appearing in the above definition of $c^+(M)$, which generalizes the definition given in the critical case. 
If $w \geq 0$ (resp., $w< 0$) this is equivalent to the vanishing of $M^{p,q}$ for which $p,q \geq 0$ (resp., $p,q < 0$), except for 
$p=q$, in which case $\sF_\infty$ must act by $-1$ (resp., $+1$) on $M^{p,p}.$

\smallskip
\subsubsection{\bf Behaviour of $c^\pm(M)$ under Tate-twists}
If the integer $n$ is even (resp., odd), so that the involution $\sF_\infty$ of $\Q(n)_{\rm B}$ is $+1$ (resp., $-1$), one has 
$M(n)^+ = M^+ \otimes (2\pi i)^n\Q$ (resp., $M(n)^+ = M^- \otimes (2\pi i)^n\Q$), so that 
$$
c^+(M(n)) = (2\pi i)^{n \dim(M^+)}c^+(M), \quad 
\mbox{(resp., $c^+(M(n)) = (2\pi i)^{n \dim(M^-)}c^-(M))$}.
$$
The case $n=1$ can be rewritten 
$$
c^-(M) = (2\pi i)^{- \dim(M^-)}c^+(M(1)).
$$

\smallskip
\subsubsection{\bf Yoshida invariants}

For $M$ a weight $w$ Betti--de Rham system of realizations over $\Q$, with coefficients in $E$, Yoshida \cite{yoshida} defines additional invariants, again in $(E \otimes \C)^*/E^*.$
The basic ones are indexed by the $p$ such that $p \leq q := w-p$ and $h^{p,q} \neq 0$, as well as by $p = -\infty.$ The invariant $c_p(M)$ is defined to be the determinant, 
computed in Betti and de~Rham rational bases of a map 
\begin{equation}
\label{eqn:yoshida}
M_{\rm B} \otimes\C \   \longrightarrow \ M_{\rm dR}/\cF^{p+1}\otimes \C \  \oplus \ M_{\rm dR}/\cF^{q} \otimes \C. 
\end{equation}
The map to the first (resp., second) factor is the projection of $M_{\rm dR}$ to $M_{\rm dR}/\cF^{p+1}$ (resp., to $M_{\rm dR}/\cF^q$), complexified, 
composed with the comparison isomorphism ${\sf comp}$ (resp., ${\sf comp} \circ \sF_\infty$). 
The map \eqref{eqn:yoshida} is an isomorphism because 
the kernel of the surjective map to the first (resp., second) factor is the sum of the $M^{a,b}$ for $a \leq p$ (resp., $a > p$, as $\sF_\infty$ exchanges $M^{a,b}$ and $M^{b,a}$). 
For $p = -\infty$, $c_p(M)$ is the determinant $\delta(M)$ of the comparison isomorphism. When $c^\pm(M)$ are defined, the last $c_p(M)$ is the product 
$c^+(M)c^-(M).$

\smallskip

Suppose we have a family of systems $M_\alpha$, and a system $M$ derived from the $M_\alpha$ by a multilinear algebraic construction. If $\dim(M_\alpha) = d_\alpha$; such constructions
correspond to representations of $\prod_\alpha \GL(d_\alpha)$ (such that the center acts by scalars not to mix weights). Yoshida proves that the invariant $c_p(M)$ of $M$, and the invariants $c^\pm(M)$ when defined, are monomials in the invariants for the $M_\alpha$. As an application of this theorem, he shows that when $\dim(M^+) = \dim(M^-)$ and 
$\dim(N^+) = \dim(N^-)$, and when $M \otimes N$ has vanishing middle Hodge number, then by Cor.\,1 to Prop.\,12 on p.\,1188 of \cite{yoshida} one has:
$$
c^+(M \otimes N) \ = \ c^-(M \otimes N). 
$$

\smallskip

To understand, it is convenient to work in a graded tensor category larger than Betti--de Rham systems over $\Q$. The objects of weight $w$ consist of 
\begin{enumerate}
\item[(a)] a $\Q$-vector space $H_{\rm B}$, with an involution $\sF_\infty$, giving a decomposition $H_{\rm B} = H_{\rm B}^+ \oplus H_{\rm B}^-$; 
\item[(b)] a $\Q$-vector space $H_{\rm dR}$, with a decreasing filtration $\cF$; 
\item[(c)] a comparison isomorphism ${\sf comp} : H_{\rm B} \otimes \C \to H_{\rm dR} \otimes \C$, 
\end{enumerate}
such that the involution $\sF_\infty$ of $H_{\rm dR} \otimes \C$, transported by ${\rm comp}$ of the involution $\sF_\infty$ of $H_{\rm B}$, transforms the filtration 
$\cF$ into a filtration that is $w$-opposite to $\cF$ (giving us a decomposition of $H_{\rm dR} \otimes \C$ into $\oplus_{p+q = w} H^{p,q}$). 

\smallskip

Given $(H_{\rm B},  \sF_\infty)$ and $(H_{\rm dR}, \cF)$ of the same dimension, the space of possible comparison isomorphisms is not empty when 
$\dim {\rm Gr}_\cF^p(H_{\rm dR}) = \ \dim {\rm Gr}_\cF^q(H_{\rm dR})$ for $p+q = w$ and either the weight is odd, or in even weight $2p$, 
$|\dim (H_{\rm B}^+) - \dim (H_{\rm B}^-)| \leq \dim {\rm Gr}_\cF^p(H_{\rm dR}).$ When this space is not empty, it is the space of complex points of the 
open orbit of $Z(\sF_\infty) \times P_{\rm dR}$ acting on $\Hom(H_{\rm B},  H_{\rm dR})$. Here, $Z(\sF_\infty)$ is the centralizer of 
$\GL(H^+) \times \GL(H^-)$ of $\sF_\infty$, and $P_{\rm dR}$ is the parabolic subgroup of $\GL(H_{\rm dR})$ respecting $\cF$. The existence of an open orbit 
is a special case of the theorem that for any reductive group, the subgroup fixed by an involution acts with finitely many orbits on flag varieties. 

\smallskip

The invariants we considered are the values at ${\sf comp}$ of functions, defined over $\Q$, and characterized up to a constant factor by their transformation law 
$$
f(p x z) \ = \ \lambda_1(p) f(x) \lambda_2(z),
$$
with $\lambda_1$ (resp., $\lambda_2$) a character of $P_{\rm dR}$ (resp., $Z(\sF_\infty)$). The invariants deduced from such functions $f$ are monomials, with exponents in $\Z$, 
in the basic invariants and $c^\pm$ when defined. In \cite{yoshida}, the attention is restricted to the $f$ being polynomial in the whole of $\Hom(H_{\rm B},  H_{\rm dR})$, which correspond
to monomials with nonnegative integral exponents. 

\smallskip

In the larger category, the invariant $c_p(M)$ can be described as being $c^+(M\otimes N)$ for $N$ of rank $2$, type $\{(-p-1,-q), (-q, -p-1)\},$ and rational comparison isomorphism.

\smallskip

\subsection{Two successive critical points}
Let us suppose that $M$ comes from a motive $\M$ and that enough conjectures are satisfied so that the $L$-function $L(\M,s)$ is defined. One has $L(\M(n),s)=L(\M,n+s)$. Suppose that $M$ and $M(1)$ are critical. When the weight $w$ of $M$ is even, $w=2p$, and that $M^{p,p} \neq 0$, 
this happens only for $w=0$, and $\sF_{\infty}$ has to act by -1 on $M^{0,0}$.

The completed $L$-function is
$$
\Lambda(\M,s) := L_{\infty}(M,s)\ L(\M,s).
$$

\begin{lem}
\label{lem:2.4.1}
If $M$ and $M$(1) are critical, then 
$$
L_{\infty} (M,1)/L_{\infty}(M,0) \ \approx \ \pi^{-\dim(M^-)} \ \pmod{\Q^*}. 
$$
\end{lem}
This follows from 
$$\Gamma_\C(s+1)/\Gamma_\C(s)\ \approx \ 1/\pi \pmod{\Q^*},
$$ 
for $s$ an integer, and 
$$
\Gamma_\R(2)/\Gamma_\R(1)\ =\ \pi^{-1}/(\pi^{-\frac{1}{2}}\cdot\ \pi^{\frac{1}{2}}) \ = \ 1/\pi.
$$

\medskip
\begin{cor}
Suppose that $M$ and $M(1)$ are critical, that $c^+(M) \approx c^-(M) \pmod{\Q^*}$, and that $L(\M,s)$ does not vanish for $s=0$ or for $s=1$. If the conjecture relating $c^+$ and $L$-values hold for $M$ and $M(1)$, then one has
$$
\Lambda (\M,1)\ \approx \ \Lambda(\M,0) \ \pmod{\Q^*}.
$$
\end{cor}

If $c^+(M)\ \approx\ c^-(M)$ then $\dim M^-$ is necessarily even, as otherwise $c^+(M)$ would be real and $c^-(M)$ imaginary. 
From \eqref{eqn:conj-Q} and Lem.\,\ref{lem:2.4.1} one has: 
$$
\frac{\Lambda (\M,0)}{\Lambda(\M,1)} \ = \ 
\frac{L_{\infty}(M,0)}{L_{\infty}(M,1)} \cdot \frac{L(\M,0)}{L(\M,1)} \ \approx \ 
\pi^{\dim(M^-)} \cdot \frac{c^+(M)}{c^+(M(1))} \ \approx \ 
i^{\dim(M^-)}\cdot \frac{c^+(M)}{c^-(M)}. 
$$
Replacing $\Q^*$  by $\E^*$, the same applies to motives with coefficients in $\E$.

\medskip

\subsection{Multilinear algebra constructions}
\label{sec:mult-lin-alg}

Suppose $(M_\alpha)_{\alpha \in A}$ is a finite family of Betti--de~Rham systems of realizations, with $M_\alpha$ of weight $w_\alpha.$ 
If $\fr = (r_\alpha)$ and $\fs = (s_\alpha)$ are systems of non-negative integers, we define 
$T_{\fr, \fs}$ to be $(\otimes M_\alpha^{\otimes r_\alpha}) \otimes (\otimes M_\alpha^{{\sf v}\otimes s_\alpha}).$ This tensor system is 
 of weight $\sum_\alpha (r_\alpha-s_\alpha) w_\alpha$. A multilinear algebra structure $\fS$ on $(M_\alpha)$ is the data of a family of morphisms 
$$
t_i : \Q(0) \ \to \ T_{\fr(i), \fs(i)}((M_\alpha)),
$$
(with the targets necessarily of weight zero) and of sub-objects of sums of tensor systems all of the same weight. 

\smallskip

Let $\fS$ be such a structure and let $\upsilon \in \{{\rm B}, {\rm dR}\}.$ We define $G(\fS)_\upsilon$ to be the algebraic subgroup of 
$\prod_\alpha \GL(M_{\alpha \, \upsilon})$
that respects the $\upsilon$-realizations $\fS_\upsilon$ of $\fS.$ One has $\Q(0)_\upsilon = \Q$ and $G(\fS)_\upsilon$ is the subgroup of 
$\prod_\alpha \GL(M_{\alpha \, \upsilon})$ which fixes
the tensors $t_{i \upsilon}(1)$ 
and stabilizes the $\upsilon$-realizations of the given subobjects.
Let the multiplicative group $\G_m$ act on $M_{\alpha \, \upsilon}$ by 
$z \mapsto \mbox{multiplication by $z^{w_\alpha}$}.$ The morphism 
$\G_m \to \prod_\alpha \GL(M_{\alpha \, \upsilon})$ defining these actions factors through $\boldsymbol{w} : \G_m \to G(\fS)_\upsilon$, the  central {\it weight morphism}. A representation $V$ of $G(\fS)_\upsilon$ is of {\it weight} $w$ if $\boldsymbol{w}(z)$ acts on $V$ by multiplication by $z^w.$ 

\smallskip

The construction which follows formalizes the notion of a {\it system $M$ constructed by multilinear algebra from $(M_\alpha)$ and $\fS.$} The construction is that of a functor $V \mapsto M^V$, compatible with direct sums, tensor products, and duals, from representations of some weight $w$ of 
$\mathrm{G}(\fS)_{\upsilon}$ to systems of realizations. Applied to the representation $M_{\alpha\, \upsilon}$, it gives back $M_{\alpha}$. This construction is a corollary of the fundamental theorem on tannakian categories. We will first explain this, and then give a more down to earth description of $V \mapsto M^V.$ 

\subsubsection{\bf Tannakian formalism giving $V \mapsto M^V$}
\label{sec:tannakian}
Let $\cT$ be the category of finite formal sums of Betti--de Rham systems of realizations of different weights. In other words, let $\cT$ be the category of systems $N$ consisting of graded vector spaces $N_{\rm B} = \oplus N_{\rm B}^w$, and $N_{\rm dR} = \oplus N_{\rm dR}^w$, and for each $w$ 
of the data turning $(N_{\rm B}^w, N_{\rm dR}^w)$ into a Betti--de Rham system of realization of weight $w$. With its natural associative and commutative 
tensor product, the category $\cT$ is tannakian, and the $\otimes$-functors $N \mapsto N_{\rm B}$ and $N \mapsto N_{\rm dR}$ are fiber functors. 

\smallskip

Let $\langle (M_\alpha) \rangle$ be the subcategory of $\cT$ generated from the $(M_\alpha)$ by iteratively taking direct sums, tensor products, duals, and 
subquotients. It is again tannakian. 
The {\it fundamental group} $\pi(\langle (M_\alpha) \rangle, \upsilon)$ is the linear algebraic group of automorphisms of 
the  $\otimes$-functor $\upsilon$ from $\langle (M_\alpha) \rangle$ to $\Q$-vector spaces. 
This fundamental group maps to $\prod \GL(M_{\alpha \, \upsilon})$ and this map identifies it with the closed subgroup of $\prod \GL(M_{\alpha \, \upsilon})$ which respects the $\upsilon$-realization of \underline{all} multilinear algebra structures on the $M_\alpha.$ Now, $\upsilon$ being a fiber functor induces an equivalence from $\langle (M_\alpha) \rangle$ to the category
of representations of $\pi(\langle (M_\alpha) \rangle, \upsilon).$ 
The group $\pi(\langle (M_\alpha) \rangle, \upsilon)$ is a subgroup of $G(\fS)_\upsilon$. If $V$ is a representation of $G(\fS)_\upsilon$, then 
$M^V$ is obtained by applying the inverse equivalence to the restriction of $V$ to $\pi(\langle (M_\alpha) \rangle, \upsilon).$ 

\smallskip
The same argument applies to systems of realizations including in addition $\ell$-adic realizations and comparison isomorphisms with Betti realization tensored with $\Q_{\ell}$.

\medskip

\subsubsection{\bf Explicating the construction of $M^V$}
Here is a more down to earth description of $V \mapsto M^V$, 
for Betti--de Rham systems of realizations,
when $\upsilon = {\rm B}.$ For each $\alpha$, the Hodge decomposition of 
$M_{\alpha \, {\rm B}} \otimes \C$ defines the action of $\G_m \times \G_m$ on $M_{\alpha \, {\rm B}} \otimes \C$ for which 
$(z_1, z_2)$ acts on $M_\alpha^{p,q}$ by multiplication by $z_1^{-p}z_2^{-q}.$ Each of the tensors $t_{i\, {\rm B}}(1)$ of the 
structure $\fS_{\rm B}$ is of type $(0,0)$ and fixed by $\sF_\infty$, being in the image of the morphism from $\Q(0)$ to a system of realizations. 
Similarly, for each subobject $N$, part of the 
multilinear algebra structure, $N_{\rm B} \otimes \C$ is sum of its $N^{a,b}$ and is stable by $\sF_\infty$. 
The action $\G_m \times \G_m \to \prod \GL(M_{\alpha\, {\rm B}} \otimes \C)$ hence factors through $G(\fS)_{\rm B}(\C)$, and 
$\sF_\infty \in \prod \GL(M_{\alpha\, {\rm B}})$ is in $G(\fS)_{\rm B}(\Q)$.

\smallskip

Let $T$ be the scheme of isomorphisms between the collection of the $M_{\alpha \, {\rm B}}$, provided with the structure $\fS_{\rm B}$, and 
the collection of the $M_{\alpha \, {\rm dR}}$, provided with the structure $\fS_{\rm dR}.$ While there are interesting cases where the scheme 
$T$ has no rational points, it is not empty: the collection of comparison isomorphisms 
$$
\fC \ := \ (M_{\alpha \, {\rm B}} \otimes \C \to M_{\alpha \, {\rm dR}} \otimes \C)_{\alpha \in A}
$$
is a complex point of $T$. As $T$ is not empty, it is a principal homogeneous space (torsor) for the right composition action of 
$G(\fS)_{\rm B}$, as well as for the left composition action of $G(\fS)_{\rm dR}.$

\smallskip

For a representation $(\sigma, V)$ of $G(\fS)_{\rm B}$, one defines
$$
M^V_{\rm B} \ := \ V, \quad \mbox{$\sF_\infty$ for $M^V$ is $\sigma(\sF_\infty).$}
$$
For the de Rham realization we define
$$
M^V_{\rm dR} \ := \ \mbox{$V$ twisted by the $G(\fS)_{\rm B}$-torsor $T$.}
$$
It is characterized, up to unique isomorphism, as being given with $ \rho : T \to {\rm Isom}(M^V_{\rm B}, M^V_{\rm dR})$, 
equivariant for the right action of $G(\fS)_{\rm B}$. Over any extension $L$ of $\Q$, a point $p \in T(L)$ defines an isomorphism
$\rho(p)$ of $L$-vector spaces from $M^V_{\rm B} \otimes L = V \otimes L$ to $M^V_{\rm dR} \otimes L,$ and 
$\rho(pg) = \rho(p) \sigma(g),$ for $g \in G(\fS)_{\rm B}$. This defines $M^V_{\rm dR} \otimes L$ up to unique isomorphism 
when $T(L)$ is not empty, and $M^V_{\rm dR}$ is defined from these $M^V_{\rm dR} \otimes L$ by descent. 
One defines the comparison isomorphism $M^V_{\rm B} \otimes \C  \to M^V_{\rm dR} \otimes \C$ to be $\rho(\fC)$.

\smallskip

The Hodge decompositions of the $M_{\alpha \, {\rm B}} \otimes \C$ define a morphism over $\C$ from 
$\G_m(\C) \times \G_m(\C)$ to $G(\fS)_{\rm B}(\C)$. In turn, this morphism defines a Hodge decomposition 
of $V\otimes \C = M^V_{\rm B} \otimes \C = \oplus(M^V)^{p,q},$ and the Hodge filtration 
$\cF^a = \oplus_{p \geq a} (M^V)^{p,q}$ over $\C$ of $M^V_{\rm dR} \otimes \C$. 

\begin{lem*}
The above filtration $\cF^\bullet$ of $M^V_{\rm dR} \otimes \C$ is rational, i.e., it is the complexification of a filtration 
$\cF_{\rm dR}$ of $M^V_{\rm dR}.$
\end{lem*}

Granting the lemma one completes the construction of $M^V$ by defining the Hodge filtration of $M^V_{\rm dR}$ to be $\cF_{\rm dR}.$ 

\begin{proof}[Proof of Lemma]
Let $S$ be a sum of $T_{\fr(i), \fS(i)}(M_\alpha)$, and let us first suppose that $V$ is a subrepresentation of $S_{\rm B}$. The
image of $V \otimes \C$ by $t \in T(\C)$ is independent of $t$. As $T$ is defined over $\Q$, this image is the complexification 
of a subspace $V_{\rm dR}$ of $S_{\rm dR}$. The morphism $T \to {\rm Isom}(V, V_{\rm dR})$ identifies $M^V_{\rm dR}$ with
$V_{\rm dR}.$ The filtration $\cF^\bullet$ is rational because it is the 
trace of the de Rham filtration of $S_{\rm dR} \otimes \C$. For $V$ a subrepresentation of $S_{\rm B}$, this exhibits $M^V$ as a 
subobject of $S$. Any $V$ can be realized as a subquotient $V_1/V_2$ of some $S_{\rm B},$ and the filtration $\cF^\bullet$ of 
$M^{V_1/V_2}_{\rm dR} \otimes \C$ is rational because it is the image of the rational filtration of $M^{V_1} \otimes \C.$
\end{proof}

\subsubsection{\bf The image of $G(\fS)_{\rm B}$ under a representation.}

\begin{prop*}
If $(\sigma, V)$ is a representation of $G(\fS)_{\rm B}$, and if $M^V$ is the corresponding system of realizations deduced from 
the data $((M_{\alpha}), \fS, (\sigma, V))$, then the image of $G(\fS)_{\rm B}$ in $\GL(M^V_{\rm B}) = \GL(V)$ is the group 
$G(\fT)_{\rm B}$ for some multilinear algebra structure $\fT$ on $M^V.$ 
\end{prop*}

This results from the theorem that any algebraic subgroup of $\GL(V)$ is the subgroup which respects some multilinear algebra 
structure $\fT$ on $V$. For an algebraic subgroup of $\GL(V)$ containing the image of $G(\fS)_{\rm B}$, such a structure 
$\fT$ is necessarily the Betti realization of a multilinear algebra stucture on $M^V$.

\medskip
\subsection{A criterion for $c^+(M)\approx_\E c^-(M)$}
\label{sec:criterion}

Let $M$ be a Betti--de Rham system of realizations, and $\fS$ be a multilinear algebra structure on $M.$ The subgroup $G(\fS)_{\rm B}$ of 
$\GL(M_{\rm B})$ respecting the Betti realization $\fS_{\rm B}$ of $\fS$ contains the Frobenius at infinity $\sF_\infty.$ Let $Z(\sF_\infty)$ be 
its centralizer. It acts on the $+1$ and $-1$ eigenspaces $M^+$ and $M^-$ of $\sF_\infty.$ Let $\chi^+$ and $\chi^-$ be the characters of 
$Z(\sF_\infty)$ which are the determinants of the action of $Z(\sF_\infty)$ on $M^+$ and $M^-.$ 
If $M$ is with coefficients in $\E$, and 
if $\E \to {\rm End}(M)$ is part of the structure $\fS$, then $G(\fS)_{\rm B}$ is contained in $\GL_\E(M_{\rm B}),$ and one similarly defines
algebraic homomorphisms $\chi_\E^+$ and $\chi^-_\E$ from $Z(\sF_\infty)$ to $\E^*:$ 
$$
\xymatrix{
Z(\sF_\infty) \ar[r] \ar[rd]_{\chi_\E^\pm} & \GL_\E(M^\pm) \ar[d]^{\det} \\ 
& \E^*
}$$
To view $\E \to {\rm End}(M)$ as a multilinear algebra 
structure, one encodes it by endomorphisms of $M$ indexed by a basis of $\E$ and viewed as morphisms $\Q(0) \to M^{\sf v} \otimes M;$ 
this uses that ${\rm End}(M) = \Hom(\Q(0), M^{\sf v} \otimes M).$ 
Let $P$ be the subgroup of $G(\fS)_{\rm B}(\C)$ which respects 
the Hodge filtration of $M_{\rm B} \otimes \C.$ 

\smallskip

As in 2.3.2, let $T$ be the scheme of isomorphisms from $(M_{\rm B}, s_{\rm B})$ to $(M_{\rm dR}, s_{\rm dR})$, and $P_{\rm dR}$ be the algebraic subgroup of $G(s)_{\rm dR}$ respecting 
$\cF_{\rm dR}$. For $\bQ$ the algebraic closure of ${\Q}$ in $\C$, $\Gal(\bQ/\Q)$ acts on set $\pi_0(T(\C))$ on connected components of $T(\C)$. Let $K_0$ be the field of definition of the connected  component $T(\C)_0$ containing the comparison isomorphism $\fC \in T(\C)$. 
It is the finite extension of $\Q$ in $\C$ corresponding to the stabilizer of $T(\C)_0 \in \pi_0(T(\C))$. Let $K_1 \subset K_0$ the field of definition of the union $P_{\rm dR} (\C)\ T(\C)_0\ Z(\sF_{\infty})(\C)$ of connected components. The most interesting case is when the union is all of $T(\C)$ (it is so when $G(s)_{\rm B}$ is connected). In that case, $K_1 = \Q$.

\smallskip
\begin{thm}
\label{thm:criterion-M-V}
Suppose that $M$ is with coefficients in $\E$, and $\E \to {\rm End}(M)$ is part of the multilinear algebra structure $\fT$. 
If the weight $w$ of $M$ is even, then suppose also that 
then $M^{w/2,w/2} = 0.$ 

\smallskip

\begin{enumerate}
\item[(i)] If the restrictions of 
$\chi^+_E$ and $\chi^-_E$ to the connected component $Z(\sF_\infty)^0$ of $Z(\sF_\infty)$ are equal, then 
$c^+(M)/c^-(M)$ is algebraic. More precisely, $c^+(M)/c^-(M) \in (\E \otimes K_0)^*.$ 
\smallskip
\item[(ii)] If $\chi^+_\E = \chi^-_\E$, 
then $c^+(M)/c^-(M) \in (\E^* \otimes K_1)^*.$ 
Moreover, if $G(\fS)_{\rm B}$ is connected, then $c^+(M)/c^-(M) \in E^*$.
\end{enumerate}
\end{thm}

\begin{proof}
Fix bases over $\E$ of $M^\pm$, and of $M_{\rm dR}/\cF^n$ where $2n$ or $2n-1$ is the weight of $M.$  Given $t \in T$, define $t^\pm$ by the diagram: 
$$
\xymatrix{
M^{\pm} \ar@{^{(}->}[d] \ar[rr]^{t^{\pm}} & & M_{\rm dR}/\cF^n\\
M_{\rm B} \ar[rr]^{t}& & M_{\rm dR}  \ar@{->}[u]
}
$$
Define algebraic functions $f^\pm : T \to \E$ as 
$f^\pm(t)  := \det(t^\pm).$  
Here, $E$ is viewed as the scheme over $\Q$ whose points with coefficients in a $\Q$-algebra $R$ is $E \otimes R,$ and 
$f^\pm$ maps $t \in T(R)$ to the determinant of $t^\pm : M^\pm \otimes R \to M_{\rm dR}/\cF^n \otimes R,$ computed with 
respect to the $E$-bases fixed above; $f^\pm(t)$ is in $E \otimes R$ and as an 
element of $(E \otimes R)/ E^*$ it is independent of choice of bases. 
 
\smallskip

Let $\chi_{\rm dR} : P_{\rm dR} \to \E^*$ be the determinant of the action of $P_{\rm dR}$ on $M_{\rm dR}/\cF^n.$
One has 
\begin{equation}
\label{eqn:c-pm-alg}
c^\pm(M) \ = \ f^\pm(\fC), 
\end{equation}
where $\fC \in T(\C)$ is the comparison isomorphism. Furthermore, for all $t \in T$, $p \in P_{\rm dR}$ and $z\in Z(\sF_\infty)$ 
we have the equivariance: 
\begin{equation}
\label{eqn:equivariance}
f^\pm(p t z) \ = \ \chi_{\rm dR}(p) f^\pm(t) \chi^\pm(z).
\end{equation}

\begin{lem}
\label{lem:f+/f-}
Under the assumption (i), on the connected component of $T(\C)$ containing $\fC$, $f^+$ and $f^-$ are generically invertible and 
$f^+/f^-$ is a constant. 
\end{lem}

\begin{proof}
By \eqref{eqn:c-pm-alg}, $f^+$ and $f^-$, with values in $\E \otimes \C,$ are invertible at the point $\fC \in T(\C)$, and hence in a Zariski neighbourhood 
of $\fC$. In this neighbourhood, it follows from \eqref{eqn:equivariance} that $f^+/f^-$ is constant on the double coset $P_{\rm dR}(\C) \fC Z(\sF_\infty)^0(\C).$
It suffices to check that this double coset contains an open neighbourhood of $\fC$. If we identify $G(\fT)_{\rm B}(\C)$ with $T(\C)$ by $g \mapsto \fC g,$ 
then this means showing that ${\rm Lie}(G(\fT)_{\rm B}(\C))$ is the sum of ${\rm Lie}(Z(\sF_\infty)(\C))$ and ${\rm Lie}(P_{\rm dR}(\C)).$
Indeed, the adjoint representation of $G(\fT)_{\rm B}$ is the Betti realization of the system of realizations of weight $0$ it defines; call it $N$. In it, 
${\rm Lie}(Z(\sF_\infty))$ is $N^+$ while ${\rm Lie}(P_{\rm dR})$ is $\cF^0.$ One concludes by observing that in any system of realizations $N$ of weight $0$, 
$N_{\rm B} \otimes \C$ is the sum of $N_{\rm B}^+ \otimes \C$ and $\cF^0.$
\end{proof}

\medskip 
\noindent{\it Proof of Thm.\,\ref{thm:criterion-M-V} (i).}
The scheme $T$ is over $\Q$. The connected component $T(\C)_0$ of $T(\C)$ containing $\fC$ is defined over 
$K_0 \subset \C,$ 
meaning that $T_{K_0}$ has a connected component $T_{K_0}^0$ with $T_{K_0}^0(\C) = T(\C)_0.$ On the scheme $T_{K_0}^0$ over $K$, $f^+/f^-$ is constant. The constant value must be in 
$(\E \otimes K_0)^*.$ 

\medskip 
\noindent{\it Proof of Thm.\,\ref{thm:criterion-M-V} (ii).}
The union $U:=P_{\rm dR} (\C) T(\C)_0 Z(\sF_\infty)(\C)$ of connected components of $T(\C)$ is defined over $K_1$, and it follows from (\ref{lem:f+/f-}) that $f^+$ and $f^-$ are generically invertible on this union. As $f^+$ and $f^-$ are defined over $\Q$, the union $U_1$ of the connected components of $U$ on which $f^+/f^-$ is constant is again defined over $K_1$. On it, $f^+/f^-$ is constant. The constant value must be in $\E \otimes K_1.$ 
\end{proof}

\medskip
\begin{remark}{\rm 
Suppose that $M$ has even weight 0, that $M^{0,0}\neq 0$, and that $\sF_\infty$ acts on $M^{0,0}$ either by $+1$, or by $-1$, so that $c^\pm(M)$ is defined. 
Let $\varphi : P_{\rm dR}\rightarrow E^*$ be the determinant character of the action of the action $P_{\rm dR}$ on $Gr^0_\cF(M_{\rm dR})$. Then, the conclusion of theorem continues to hold, 
under the additional assumption, for (i), that $\varphi$ is trivial on the connected component $P^0_{\rm dR}$ of $P_{\rm dR}$, and for (ii), that $\varphi$ is trivial. This is a special case of the results of the next section, and the proof is essentially the same.
}\end{remark}

\medskip

\begin{remark}{\rm 
In all the examples that we consider in the next section, $G(\fS)^0$ is reductive. 
When this is the case, one could have argued that 
$P_{\rm dR} \cap G(\fS)^0(\C)$ is a parabolic subgroup, and that $Z(\sF_\infty) \cap G(\fS)^0$ is the fixed points of an involution, to obtain 
that on any connected component of $T(\C)$, some double coset for the left action of $P_{\rm dR}(\C)^0$ and the right action of 
$Z(\sF_\infty)^0(\C)$ is dense. 
}\end{remark}

\bigskip

\subsection{Multiplicative relations between determinants of periods}
\label{sec:Multiplicative relations}

\subsubsection{} 

As in \ref{sec:mult-lin-alg}, let $(M_{\alpha})_{\alpha \in A}$ be a finite family of Betti--de Rham realizations, let $\fS$ be a multilinear structure, and define $G(s)_{\rm B}$ to be its automorphism group in the Betti realization.

\smallskip

Let $(V_i)$ be a finite family of representations of $G(s)_{\rm B}$ of weights $w_i$ on $E$-vector spaces. The $M_i := M^{V_i}$ are with coefficients in $E$, and one can attach to each of them the following determinants of periods, in $(\E^* \otimes \C)/E^*$ :

\begin{enumerate}
\item[(a)] $c^{\pm}_E(M_i).$ When $w_i$ is even, $w_i=2p$, they are defined only under the condition that $\sF_{\infty}$ acts on $(M^{V_i})^{p,p}$ either by $+1$, or by $-1$.
\smallskip
\item[(b)] The determinant $\delta_E(M_i)$ of the comparison isomorphism $M_{i, {\rm B}} \otimes \C \xrightarrow{\sim} M_{i, {\rm dR}} \otimes \C$, computed using $E$-bases of 
$M_{i, {\rm B}}$ and $M_{i, {\rm dR}}.$
\end{enumerate}

\smallskip

To each of those is associated characters $\chi^\pm : Z(\sF_{\infty})\to E^*$ and a character $\varphi : P_{\rm dR}\to \E^*.$ The character $\chi^\pm$ of $Z(\sF_{\infty})$ is respectively the determinant of the action of $Z(\sF_{\infty})$ on $M^+_i$ and $M^-_i.$ The character of $P_{\rm dR}$ is the determinant of the action of $P_{\rm dR}$ on $M_{i_{\rm dR}}/\cF^p_{\rm dR}$, where this quotient of $M_{i_{\rm dR}}$, complexified, is the target used in the definition of the respective $c^{\pm}$ or $\delta$.

\smallskip

Let now $m$ be a monomial (with exponents in $\Z$) in the $c^{\pm}_E (M_i)$ and $\delta(M_i)$, and let $\chi_m$ and $\varphi_m$ be the corresponding product of powers of the associated characters of $Z(\sF_\infty)$ and $P_{\rm dR}$. With the same definition of $K_0$ and $K_1$ as before, one has

\smallskip

\begin{thm}
\label{thm:fancy}
The image of $m$ in $(E^* \otimes \C)/(E^* \otimes K^0)^*$ (resp., in $(\E^* \otimes \C)/(E\otimes K^1)^*$) is uniquely determined the restriction of $\chi_m$ and $\varphi_m$ to $Z(\sF_\infty)^0$ and $P^0$ (resp., by $\chi_m$ and $\varphi_m$). Special cases: If $G(s)_{\rm B}$ is connected, the image of $m$ in 
$(\E^* \otimes \C)/E^*$ is uniquely determined by $\chi_m$ and $\varphi_m$. If further $\chi_m$ and $\varphi_m$ are trivial, $m$ is in $\E^*$.
\end{thm}

\medskip

The proof is the same as before: $m$ is the value at $\fC$ of a function $f$ on $U \subset T$, with values in $\E$, with the covariance property 
$$
f(p t z)=\varphi_m (p) f(t) \chi_m (z).
$$

\bigskip
\subsubsection{\bf Grothendieck period conjecture.}

If the $M_{\alpha}$'s come from motives $\M_\alpha$'s, and if $\fS$ consists of all multilinear structures on the motives $\M_\alpha$, so that $G(\fS)$ is the 
``motivic Galois group" of the family of $\M_\alpha$, Grothendieck conjectures that $\fC$ is Zarisky dense in $T$.
This conjecture implies that if monomials $m(j)$ as in Sect.\,\ref{sec:Multiplicative relations} 
give rise to multiplicatively independent characters $(\chi (j), \varphi(j)),$ the $m(j)$ are algebraically independent. 

\medskip

\begin{ex}
Let $M$ be a Betti--de Rham system of realizations of dimension 2, weight $w$ and two nonzero Hodge numbers $h^{pq} = h^{qp} = 1.$ We give it no multilinear structure, so that $G_{\rm B} := G_{\rm B}(\phi)$ is $\rm GL(M_{\rm B}).$ In a suitable basis $e^+,\ e^-$ of $M_{\rm B},$ $\sF_{\infty}$ is the diagonal matrix ${\rm diag}(+1, -1)$ and its centralizer $Z(\sF_{\infty})$ is, in that basis, the group of diagonal matrices. The characters of $Z(\sF_{\infty})$ are the ${\rm diag}(a, b) \mapsto a^m b^n,$ for integers $m, n.$ Suppose $p > q.$ 
In a basis $(e, f)$ of $M_{\rm dR},$ with $e$ spanning $\cF^p_{\rm dR}, P_{\rm dR}$ is the group of upper triangular matrices. Its characters are 
$$
\begin{pmatrix} a & b \\0 & d \end{pmatrix} \longmapsto a^r\, d^s
$$
for integers $r,s$. If $V$ is an irreducible representation of $G_{\rm B}$, the corresponding characters of $Z(\sF_{\infty})$ and $P_{\rm dR}$ are hence given by four integers $(m, n, r, s).$ Expressing that the scalar matrices are in both $Z(\sF_{\infty})$ and $P$, one finds that $m+n = r+s.$
For the representation $M_{\rm B}$ itself, one finds that the characters corresponding to $c^+, c^-$ and $\delta$ are respectively given by the $4$-tuples 
$$
c^+ \ :\ (1,0,0,1), \quad \ 
c^-  \ :\ (0,1, 0, 1), \quad \ 
\delta \ :\ (1,1,1,1).
$$
The $\Z$-span of these three $4$-tuples is the additive group of all integral $4$-tuples $(m, n, r, s)$ for which $m+n = r+s.$ It follows that for any irreducible representation $V$ of 
$\rm GL(M_{\rm B}),$ $c^+ (M^V),$ $c^-(M^V),$ and $\delta(M^V)$ are, mod $\Q^*,$ monomials in $c^+(M),$ $c^-(M),$ and $\delta(M).$ Computing the relevant $m, n, r, s,$ one finds the formula of \cite[Prop.\,7.7]{deligne} computing the $c^{\pm}(\Sym^n M).$

\smallskip

If $M$ comes from a motive $\M$ and if ${\sf s}$ is all multilinear structures of $M.$ The group $G_{\rm B} (\C)$ is at least as big as a Cartan normalizer. Indeed, it contains the $\G_m \times \G_m$ giving the Hodge decomposition for $M_{\rm B }\otimes \C,$ as well as $F_{\infty}$ which permutes the two factors $\G_m.$ If it is strictly bigger, it must be the full $\rm GL.$ In this non-CM case, Grothendieck period conjecture implies that $c^+(M)$ and $c^-(M)$ are algebraically independent.

\smallskip

From the automorphic point of view, we are here considering classical holomorphic new cuspforms of weight $k \geqslant 2.$ Such forms give rise to $M$ as above, with 
$h^{k-1, 0} = h^{0, k-1} = 1.$

\end{ex}

\bigskip

\section{Examples}
\label{sec:examples}

\subsection{Tensor product motives over $\Q$}
\label{sec:tensor-product-motives}

Let $M'$ and $M''$ be Betti--de Rham systems of realizations over $\Q$ with coefficients in $\E,$ and $M$ their tensor product over $E.$

\begin{prop}
\label{prop:tensor-product-even-even}
Suppose that $M'$ and $M$ have vanishing middle Hodge numbers. For $M''$, let 
$\delta'' = \dim(M''^+) - \dim(M''^-).$ Then 
$$
\frac{c^+(M)}{c^-(M)} \ \approx_\E \  \left(\frac{c^+(M')}{c^-(M')}\right)^{\delta''}. 
$$
In particular, if we assume $\dim(M^{''+}) = \dim(M^{''-})$, which is the case if $M''$ also has vanishing middle Hodge number, then 
$c^+(M) \approx_\E c^-(M).$
\end{prop}

This result is due to Yoshida \cite{yoshida}, and our Sect.\,\ref{sec:criterion} formalizes his argument. In Sect.\,\ref{sec:tensor-product-motives-general}, we will show that the same result holds for systems of realizations over a totally real field. 

\begin{proof}
As multilinear structures over $M'$ and $M''$, we take only the structures of $E$-modules. The group $G(\fS)_{\rm B}$ to consider is 
$\GL_E(M'_{\rm B}) \times \GL_E(M''_{\rm B})$, and $M$ corresponds to the representation $M'_{\rm B} \otimes M''_{\rm B}$. The centralizer of 
$\sF_\infty$ is the product of its centralizers in $\GL_E(M'_{\rm B})$ and $\GL_E(M''_{\rm B}).$ The assumption of vanishing middle Hodge number of 
$M'$ implies that its $E$-dimension is even, say, $2n'$; one has 
$\dim(M'^+) = \dim(M'^-) = n'$. For $M''$, define
$$
n''(+) := \dim(M''^+), \quad n''(-) \ := \ \dim(M''^-),
$$
hence $\delta'' = n''(+) - n''(-).$
If $V$ and $W$ are vector spaces (over $E$), then the determinant $\det(A \otimes B)$ of the action of $(A,B) \in \GL(V) \otimes \GL(W)$ 
on $V \otimes W$ is given by $\det(A)^{\dim(W)} \det(B)^{\dim(V)}.$ From the isomorphisms 
$$
M^+ \ = \ M'^+ \otimes M''^+ \ \oplus \ M'^- \otimes M''^-, \quad M^- \ = \ M'^+ \otimes M''^- \ \oplus \ M'^- \otimes M''^+,
$$
we get for the determinant of the action of $Z(\sF_\infty)$
\begin{eqnarray*}
\chi_M^+ & = & (\chi_{M'}^+)^{n''(+)} (\chi_{M'}^-)^{n''(-)} \cdot (\chi_{M''}^+ \chi_{M''}^-)^{n'} \\ 
\chi_M^- & = & (\chi_{M'}^+)^{n''(-)} (\chi_{M'}^-)^{n''(+)} \cdot (\chi_{M''}^+ \chi_{M''}^-)^{n'}, 
\end{eqnarray*}
so that $\chi_M^+/\chi_M^- = (\chi_{M'}^+/\chi_{M'}^-)^{\delta''}.$ 
We deduce the following relation from Thm.\,\ref{thm:fancy}:
\begin{equation}
\label{eqn:tensor-product-even-odd}
\frac{c^+(M)}{c^-(M)} \ \approx_\E \  \left(\frac{c^+(M')}{c^-(M')}\right)^{\delta''}. 
\end{equation}
It remains to check that both sides correspond to the same character of $P_{\rm dR}$, which is the product of the 
$P_{\rm dR}$ for $\GL(M'_{\rm dR})$ and $\GL(M''_{\rm dR}).$ Indeed, as $M$ as well as $M'$ have vanishing middle Hodge numbers, 
the relevant characters are trivial. 

\smallskip

Furthermore, if we assume $\delta'' = 0$, then the second assertion in the proposition trivially follows from the first.  
\end{proof}

\medskip



\medskip
\subsubsection{}
\label{sec:odd-odd}
Let us now assume only that $c^+$ and $c^-$ are defined for $M'$, $M''$, and $M$, and that both $M'$ and $M''$ have nonvanishing 
middle Hodge numbers. Let $2p'$ and $2p''$ be the weights of $M'$ and $M''$, and define $p = p'+p''.$ By assumption, 
$\sF_\infty$ acts by a sign, i.e., by either $+1$ or $-1$, on $M'^{p'p'};$ call $s(M')$ this sign. Similarly for $M''.$ The assumption that 
$c^\pm(M)$ are defined implies that $M'^{p'p'} \otimes M''^{p''p''} \simeq M^{pp},$ and that $s(M) = s(M')s(M'').$ With the notations 
of Prop.\,\ref{prop:tensor-product-even-even}, one has $\delta'' = s(M'') \dim(M''^{p''p''}).$ Defining similarly $\delta'$ for $M'$, we get the following proposition:

\begin{prop}
\label{prop:tensor-product-odd-odd}
Under the assumptions of \ref{sec:odd-odd}, one has: 
$$
\frac{c^+(M)}{c^-(M)} \ \approx_\E \  \left(\frac{c^+(M')}{c^-(M')}\right)^{\delta''}\left(\frac{c^+(M'')}{c^-(M'')}\right)^{\delta'}. 
$$
\end{prop}

Concerning the relevant characters of $Z(\sF_\infty)$, the proof is the same, but the relevant characters of $P_{\rm dR} = P'_{\rm dR} \times P''_{\rm dR}$ 
are now nontrivial. For $c^+(M)/c^-(M)$, it is given by $\det {\rm Gr}^p_\cF(M_{\rm dR})^{s(M)}.$ Similarly for
$c^+(M')/c^-(M')$ and $c^+(M'')/c^-(M'').$ One concludes by using that
\begin{eqnarray*}
\det {\rm Gr}^p_F(M_{\rm dR})^{s(M)} 
& = & 
\left( \det {\rm Gr}^{p'}_\cF(M'_{\rm dR})^{\dim(M''^{p''p''})} \cdot \det {\rm Gr}^{p''}_\cF(M''_{\rm dR})^{\dim(M'^{p'p'})}   \right)^{s(M)} \\ 
& = & 
\left[ \det {\rm Gr}^{p'}_\cF(M'_{\rm dR})^{s(M')} \right]^{\delta''} \cdot 
\left[ \det {\rm Gr}^{p''}_\cF(M''_{\rm dR})^{s(M'')} \right]^{\delta'}. 
\end{eqnarray*}

\medskip
\subsubsection{}
As explained in \cite{bhagwat-raghuram-mrl}, Prop.\,\ref{prop:tensor-product-even-even} 
gives a direct motivic explanation, via \cite{deligne}, of the main results of \cite{harder-raghuram-book} when the base field is $\Q$. 

It is worth noting that if $M'$ (resp., $M''$) is the motive of even (resp., odd) dimension, say, $2n'$ (resp., $2n''+1$), conjecturally attached to a cohomological cuspidal automorphic representation $\sigma'$ (resp., $\sigma''$)  of $\GL(2n')$ (resp., $\GL(2n''+1)$), then because the Langlands parameter of the representation at infinity is regular, one has $\delta' = 0$ and 
$\delta'' = \pm 1$. This explains why in \cite{harder-raghuram-book}, for the theorem 
on ratios of $L$-values for a Rankin--Selberg $L$-function $L(s, \sigma' \times \sigma'')$ for $\GL(2n') \times \GL(2n''+1)$, one sees the relative period $\Omega(\sigma')$ of $\sigma'$ 
or its reciprocal; the relative period $\Omega(\sigma')$ corresponds to the ratio $c^+(M')/c^-(M')$ in the right hand side of Prop.\,\ref{prop:tensor-product-even-even}, and whether it or its reciprocal appears is dictated by $\delta'' = \pm 1$.

\medskip
\subsubsection{Example: The tensor product of two rank-two motives-I}
\label{sec:prod-rank-two-motives-1}
For $i = 1,2$, let $\varphi_i \in S_{k_i}(N, \omega_i)$ be a primitive holomorphic cuspidal modular form of weight $k_i \geq 2$, for $\Gamma_0(N)$, and nebetypus character 
$\omega_i$; primitive means that it is an eigenform, a newform, and normalized $a_1(\varphi_i) = 1.$ Assume (only for convenience) that $k_1$ and $k_2$ are even. 
Let $E$ be a number field which is large enough to contain all the Fourier 
coefficients of $\varphi_1$ and $\varphi_2$. Let $M_i = M(\varphi_i)$ be the rank-two motive over $\Q$ with coefficients in $E$ associated to $\varphi_i$ (see \cite{scholl}). 
Put 
$M = M_1 \otimes M_2.$ Assume that $k_1 \neq k_2$. Then $M_1$, $M_2,$ and $M$ have vanishing middle Hodge numbers; furthermore, the $L$-function of $M$, which is the 
degree-4 Rankin--Selberg $L$-function, has at least two critical points. 
We know from Prop.\,\ref{prop:tensor-product-even-even} that $c^+(M) \approx_\E c^-(M).$ This is compatible
with the statement that for the completed Rankin--Selberg $L$-function $L(m,\varphi_1 \times \varphi_2) \approx_E L(m+1,\varphi_1 \times \varphi_2),$ for $\min\{k_1,k_2\} \leq m <m+1 < \max\{k_1,k_2\}$, which follows from Shimura \cite[Thm.\,4]{shimura-mathann}. See also Blasius \cite{blasius-appendix}. This example also works Hilbert modular forms.

\medskip
\subsubsection{}
When the base field is a general number field the situation is more delicate. See below for results that hold for systems of realizations over number fields $L$ which 
require additional arguments: as is the case for 
$L$-functions, $c^+$ and $c^-$ are defined by first restricting from $L$ to $\Q$, but this operation does not commute with taking a tensor product or 
with the construction of group of automorphisms $G(\fS).$

\medskip
\subsection{Artin Motives}
\label{sec:artin}

Let us review Artin motives, which will be used in the next section. 
If, in Grothendieck's definition of motives, one considers only varieties of dimension zero, one obtains the Artin motives. 
They give rise to the Artin $L$-functions. 
An equivalent, more down to earth definition is that an Artin motive is a representation of $\Gal(\bQ/\Q)$ on a $\Q$-vector space, 
for $\bQ$ an algebraic closure of $\Q$. This category does not depend on a choice of $\bQ$. 
Let us take for $\bQ$ the algebraic closure of $\Q$ in $\C.$ The Betti--de Rham system of realizations attached to 
the representation $V$ is then given as follows: 
\begin{itemize}
\item $V_{\rm B}$: the vector space $V$; $\sF_\infty$ is the action of complex conjugation. 

\item $V_{\rm dR}$: $(V \otimes \bQ)^{\Gal(\bQ/\Q)},$ purely of Hodge filtration $0.$ 

\item Comparison isomorphism: the inclusion of $V_{\rm dR}$ in $V \otimes \bQ$ extends to an isomorphism of $V_{\rm dR} \otimes \bQ$ to $V \otimes \bQ$, 
which one complexifies. 
\end{itemize}

\medskip
\noindent{\bf Example:} 
Let $L$ be a finite extension of $\Q.$ Then $M := H^0(\Spec \,L)$ of Sect.\,\ref{sec:Betti--deRham} is an Artin motive. The set of complex points of 
$\Spec \,L$ is the set $\Hom(L,\C)$ of complex embeddings of $L$; $M_{\rm B}$ is hence $\Q^{\Hom(L,\C)}$, the sum of copies of $\Q$ indexed by 
$\Hom(L,\C).$ The de Rham cohomology $M_{\rm dR}$ is $L.$ That $L \otimes \C$ is a direct sum of copies of $\C$ indexed by $\Hom(L,\C)$ gives
the comparison isomorphism. This turns $H^0(\Spec \,L)$ into a ring in the category of Artin motives. (The reader is also referred to \cite[Sect.\,6]{deligne}.)

\medskip
\begin{prop}
\label{prop:artin-motives}
The functor from Artin motives (i.e., representations $V$ of $\Gal(\bQ/\Q)$) to Betti--de Rham system of realizations is fully faithful. 
\end{prop}

\begin{proof}
One recovers the action of $\Gal(\bQ/\Q)$ on $V = V_{\rm B}$ by embedding $V_{\rm B}$ in $V_{\rm dR} \otimes \bQ$ by 
the comparison isomorphism, and by restricting to $V_B$ the action of $\Gal(\bQ/\Q)$ on $V_{\rm dR} \otimes \bQ$. 
\end{proof}

\medskip
\begin{remark}
\label{rem:artin-motives}
Let $M$ be the Betti--de Rham system defined by the representation $V$ of $\Gal(\bQ/\Q)$, and let $\Gal$ be the image of $\Gal(\bQ/\Q)$ in 
$\GL(V)$. From the Tannakian point of view of Sect.\,\ref{sec:tannakian}, what Prop.\,\ref{prop:artin-motives} tells us is that if on $M$ one puts all
the multilinear structures it admits, the corresponding $G(\fS)$ is $\Gal.$ {\it Special case:} if $M = H^0(\Spec \,L),$ where $L$ is a finite extension of $\Q$, and if $N$ is the normal closure of $L$ in $\bQ \subset \C$, 
generated by the images of the embeddings of $L$ in $\bQ$, one has $G(\fS) = \Gal(N/\Q).$
\end{remark}

\medskip
\subsection{Systems of Betti--de Rham realizations over a number field $L$}
\label{sec:over-L}

The motivating example is the cohomology $H^n(X)$ of a projective nonsingular variety $X$ over $L$: for such a variety, if $\sigma$ is an embedding of 
$L$ into an algebraic closure $C$ of $\R$, $X(C)$ is defined and so is $H^n(X(C), \Q).$ One also has the de Rham cohomology $H_{\rm dR}^n(X)$, a 
vector space over $L.$

\medskip
\subsubsection{}
\label{sec:realizations_over_number_field_L}

A weight $w$ Betti--de Rham system of realization over $L$ consists of the following data $({\rm i}_L)-({\rm iii}_L)$ obeying $({\rm iv}_L)$:

\smallskip
\begin{enumerate}
\item[$({\rm i}_L)$] For $\sigma$ an embedding of $L$ into an algebraic closure $C$ of $\R$, a $\Q$-vector space $M_\sigma$, depending functorially
on $C.$ If for $C$ we take the standard algebraic closure $\C$ of $\R$, the functoriality in $C$ amounts to an involutive system  
$\sF_\infty : M_\sigma \to M_{c \sigma}$ of isomorphisms.  If $\sigma$ is real, i.e., when $\sigma = c \sigma$, the involution $\sF_\infty$ of $M_\sigma$ defines
a decomposition $M_\sigma = M_\sigma^+ \oplus M_\sigma^-$ into the $\pm 1$ eigenspaces.  

\smallskip
\item[$({\rm ii}_L)$] An $L$-vector space $M_{\rm dR}$, endowed with a finite decreasing filtration $\cF_{\rm dR}$, the Hodge filtration. 

\smallskip
\item[$({\rm iii}_L)$] For $\sigma$ as in $({\rm i}_L)$, a comparison isomorphism $M_\sigma \otimes_\Q C \to M_{\rm dR} \otimes_{L,\sigma} C$, which is 
functorial in $C$. 

\smallskip
\item[$({\rm iv}_L)$] One requires that for $\sigma$ as in $({\rm i}_L)$, if the complexified Hodge filtration is transported by the comparison isomorphism 
to a filtration, still denoted $\cF$, of $M_\sigma \otimes_\Q C$, it and its complex conjugate are $w$-opposite.
\end{enumerate}

By $({\rm iv}_L)$, for each $\sigma$, $M_\sigma$ is endowed with a Hodge structure of weight $w$. Its Hodge numbers 
$h_\sigma^{p,q} = \dim_L {\rm Gr}_{\cF}^p(M_{\rm dR})$ are independent of $\sigma.$

\medskip

The restriction of base field, from $L$ to $\Q$, of such a system is defined by:
\begin{eqnarray*}
\Res_{L/\Q}(M)_{\rm B} \ & := & \bigoplus_{\sigma : L \to \C} M_\sigma, \\
\Res_{L/\Q}(M)_{\rm dR} & := & M_{\rm dR} \ \mbox{viewed as a $\Q$-vector space.}
\end{eqnarray*}
One has 
$\Res_{L/\Q}(M)_{\rm dR} \otimes_\Q \C \ \simeq \ \oplus_{\sigma : L \to \C} \, M_{\rm dR} \otimes_{L,\sigma}\C,$ and the comparison isomorphism for 
$\Res_{L/\Q}(M)$ is the sum over all $\sigma: L \to \C$ of the comparison isomorphisms of $M.$ 

\medskip

One defines $c^\pm(M)$ to be $c^\pm(\Res_{L/\Q}(M))$. This parallels the fact that when $M$ comes from a motive $\M$ with an $L$-function 
$L(\M,s)$ (which is defined as an Euler product over all places of $L$), one has 
$$
L(\M,s) \ = \ L(\Res_{L/\Q}(\M),s).
$$

\medskip

In the case of the motivating example, if for $X \to \Spec\,L$, a variety over $L$, one defines $\amalg_{L/\Q}(X)$ to be the variety over $\Q$ as 
$X \to \Spec\,L \to \Spec\,\Q$, then one has 
$$
\Res_{L/\Q}(H^n(X)) \ = \ H^n(\amalg_{L/\Q}(X)). 
$$
For the Betti realization, $\amalg_{L/\Q}(X)(\C)$ is the disjoint union over embeddings $\sigma : L \to \C$ of the $\C$-points of $X_\sigma := X \times_{L, \sigma} \C$; 
hence, $H^n(\amalg_{L/\Q}(X)(\C), \Q) = \oplus_{\sigma : L \to \C} H^n(X_\sigma(\C), \Q).$

\medskip
\subsubsection{Description as module}
\label{sec:realizations_over_L_modules}

The functor $\Res_{L/\Q}$ identifies systems of realizations over $L$ with systems of realizations $M$ over $\Q$ provided with a module structure 
over $H^0(\Spec\,L)$. The inverse functor is described as follows. The module structure of $M_{\rm B}$ over $\Q^{\Hom(L,\C)}$ gives a decomposition 
of $M_{\rm B}$ into pieces $M_\sigma$ indexed by the embedding of $L$ into $\C.$ The module structure of $M_{\rm dR}$ over $L$ turns $M_{\rm dR}$ 
into the required de Rham realization. 

\medskip

Let $M$ be a system of realizations over $L$. With the notations as in \ref{rem:artin-motives}, if on $\Res_{L/\Q}(M)$ and 
$H^0(\Spec\,L)$ we take as multilinear structures all possible structures on $H^0(\Spec\,L)$ and the module structure of 
$\Res_{L/\Q}(M)$, the resulting group $G(\fS)_{\rm B}$ is an extension of $\Gal(N/\Q)$ by $\prod_\sigma \GL(M_\sigma)$, the 
product is over all embeddings $\sigma : L \to \C,$ where $N$ is the normal closure of $L$ in $\bQ \subset \C$. 
The group elements above $\alpha \in \Gal(N/\Q)$ are the systems 
of isomorphisms $M_\sigma \to M_{\alpha \sigma}.$ 

\medskip
\subsubsection{A remark on Hodge numbers}
\medskip

As before, a system with coefficients in $E$ is a system $M$ provided with an $E$-module structure $E \to \End(M).$ A crucial complication is that 
the $L \otimes E$-modules ${\rm Gr}^p(M_{\rm dR})$ are usually not free. A consequence is that when $M$ comes from a motive and that 
for each $\sigma: L \to \C$ a corresponding $L$-function $L^\sigma$ is defined, the factors $L^\sigma_v$ at infinite places $v$ of $L$ depend on 
$\sigma.$ But their product over $v$ is independent of $\sigma.$  Very briefly, a weight $w$ Betti--de Rham system of realization $M$ over a number field $L$ with coefficients in a number field $E$ consists of: 
\smallskip
\begin{itemize}
\item (Betti) For each embedding $\sigma$ of $L$ into $\C$, an $E$-vector space $M_\sigma$; together with a Hodge decomposition 
$M_\sigma \otimes \C = \oplus_{p+q = w} M_\sigma^{p,q}$ of $E \otimes \C$-modules; 
\item (Real Frobenius) an involutive isomorphism $\sF_\infty : M_\sigma \to M_{c \sigma}$ of $E$-vector spaces which upon complexification 
maps $M_\sigma^{p,q}$ to $M_{c \sigma}^{q, p};$
\item (de Rham) an $E \otimes L$-module $M_{\rm dR}$, endowed with a Hodge filtration $\cF_{\rm dR}$; and
\item (Comparison) a comparison isomorphism $M_\sigma \otimes_\Q \C \to M_{\rm dR} \otimes_{L,\sigma} \C$ of $E \otimes \C$-modules, 
inducing 
$M_\sigma^{p,q} = {\rm Gr}^p \cF_{\rm dR} \otimes_{L, \sigma} \C.$
\end{itemize}
In the Hodge decomposition, the summand $M_\sigma^{p,q}$ being an $E \otimes \C$-module amounts to specifying a $\C$-vector space 
$M_{\sigma, \tau}^{p,q}$ for every embedding $\tau : E \to \C.$ Define the Hodge numbers of $M$ as: 
\begin{equation}
\label{eqn:hodge-numbers}
h_{\sigma, \tau}^{p,q} \ = \ \dim_\C(M_{\sigma, \tau}^{p,q}).
\end{equation}

\medskip
\begin{prop}
\label{prop:strong-purity}
With notations as above, the Hodge number 
$h_{\sigma, \tau}^{p,q}$ depends only on the restriction of $\sigma$ to the largest CM or totally real subfield of the base field $L$, and 
only on the restriction of $\tau$ to the largest CM or totally real subfield of the coefficient field $E$.
\end{prop}

\begin{proof}
From $\sF_\infty$, one has
\begin{equation}
\label{eqn:hodge-number-complex-conj}
h_{\sigma, \tau}^{p,q} \ = \ h_{c\sigma, \tau}^{q,p}.
\end{equation}
Since $L$ and $E$ are number fields, $\sigma$ and $\tau$ take values in the algebraic closure $\bQ$ of $\Q$ in $\C$. There is an action 
of $L \otimes E$ on ${\rm Gr}^p \cF_{\rm dR}$, and $h_{\sigma, \tau}^{p,q} = \dim_{\bQ}({\rm Gr}^p \cF_{\rm dR} \otimes_{L, \sigma} \bQ);$
hence, for any $\gamma \in \Gal(\bQ/\Q)$ one has 
\begin{equation}
\label{eqn:hodge-number-gal}
h_{\sigma, \tau}^{p,q} \ = \ h_{\gamma \sigma, \gamma \tau}^{p,q}.
\end{equation}
From \eqref{eqn:hodge-number-complex-conj} and \eqref{eqn:hodge-number-gal}, 
$$
h_{\sigma, \tau}^{p,q} \ = \ 
h_{\gamma \sigma, \gamma \tau}^{p,q} \ = \ 
h_{c \gamma \sigma, \gamma \tau}^{q,p} \ = \ 
h_{\gamma^{-1} c \gamma \sigma, \tau}^{q,p} \ = \ 
h_{c \gamma^{-1} c \gamma \sigma, \tau}^{p,q}.
$$
Hence, $h_{\sigma, \tau}^{p,q} = h_{\eta\sigma, \tau}^{p,q}$ for all $\eta$ fixing the maximal CM subfield of $\bQ.$ The Galois set 
$\Hom(L, \bQ)$, modulo the action of all such $\eta$ corresponds to $\Hom(L_1, \bQ)$ for 
the maximal CM or totally real subfield $L_1$ of $L$. Applying \eqref{eqn:hodge-number-gal} to $\gamma = c$ we get 
\begin{equation}
\label{eqn:hodge-number-complex-conj-2}
h_{\sigma, \tau}^{p,q} \ = \ h_{\sigma, c\tau}^{q,p}, 
\end{equation}
and the same argument applies to $\tau.$
\end{proof}

\medskip
\begin{cor}
\label{cor:strong-purity}
If $L$ or $E$ admits a real embedding, one has 
$$
h_{\sigma, \tau}^{p,q} \ = \ h_{\sigma, \tau}^{q,p}.
$$
\end{cor}

\begin{proof}
If $L$ (resp., $E$) has a real embedding, then its largest CM or totally real subfield of $L$ is necessarily totally real, so that 
$h_{\sigma, \tau}^{p,q} \ = \ h_{c\sigma, \tau}^{q,p}$ (resp., $h_{\sigma, \tau}^{p,q} \ = \ h_{\sigma, c\tau}^{q,p}$), and one 
applies \eqref{eqn:hodge-number-complex-conj} (resp., \eqref{eqn:hodge-number-complex-conj-2}). 
\end{proof}

\medskip
\begin{cor}
\label{cor:sigma-tau-indep-L-real-emb}
If $L$ has a real embedding $\sigma_0$, and suppose the weight is even, say $2p$, and the involution $\sF_\infty$ of $M_{\sigma_0}$ acts on 
$M_{\sigma_0}^{p,p}$ by either $+1$ or $-1$, then 
$$
\sum_{r<s} h_{\sigma, \tau}^{r,s} \ \ \ {\rm and} \ \ \ \sum_{r \leq s} h_{\sigma, \tau}^{r,s}
$$
are independent of $\sigma$ and $\tau.$

In the case of an even weight $2p$ (resp, odd weight $2p-1$), the conclusion can be rephrased as ${\rm Gr}^p(M_{\rm dR})$ and $\cF^{p+1}(M_{\rm dR})$ 
(resp., $\cF^p(M_{\rm dR})$) is a free $L \otimes E$-module. 
\end{cor}

\begin{proof}
In the even weight $2p$ case, for the Hodge structure of $M_{\sigma_0}$, $M_{\sigma_0}^+$ (resp., $M_{\sigma_0}^-$), complexified, maps isomorphically to 
$(M_{\rm dR}/\cF^p) \otimes_{L, \sigma_0} \C$ (resp., $(M_{\rm dR}/\cF^{p+1}) \otimes_{L, \sigma_0} \C$). Since $M_{\sigma_0}^\pm$ are free $E \otimes \C$-modules, 
so are 
$(M_{\rm dR}/\cF^p) \otimes_{L, \sigma_0} \C$ and $(M_{\rm dR}/\cF^{p+1}) \otimes_{L, \sigma_0} \C$, meaning that 
$\sum_{r<p} h_{\sigma_0, \tau}^{r,s}$ and $h_{\sigma_0, \tau}^{p,p}$ are independent of $\tau$; one concludes by \eqref{eqn:hodge-number-gal}. 
The argument is similar in the odd weight case. 
\end{proof}

\medskip
\subsubsection{Highest weights}
Prop.\,\ref{prop:strong-purity} has an analogue on automorphic side: in \cite{raghuram-imaginary}, a condition called {\it strong-purity} is identified as a necessary condition on a dominant integral 
weight $\lambda$ to support cuspidal cohomology for $\GL(n)$ over a number field $L.$ Roughly speaking it says that such a highest weight $\lambda$ is the base-change from $L_1$ of a pure 
dominant integral weight for $\GL(n)/L_1$; see \cite[Prop.\,2.6]{raghuram-imaginary}; whereas the discussion in {\it loc.\,cit.}\ is for a totally imaginary base field $L$, 
it applies {\it mutatis mutandis} for
a general number field.

\medskip
\subsection{Systems of Betti--de Rham realizations over a totally imaginary field}
\label{sec:over-tot-imag}

Let $M$ be a Betti--de Rham system of realizations over a totally imaginary field $K$, with coefficients in a number field $E$. We assume vanishing middle
Hodge number. Our aim is to express $c^+_E(M)/c^-_E(M)$, an element of $(\C \otimes E)^*/E^*$, 
in terms of the invariant $\delta(K)$ of $K$ defined in \ref{sec:delta-K} below. This invariant is the image in 
$(\C \otimes E)^*/E^*$ of an element of order dividing $2$ of $\C^*/\Q^*$, by the injection:
$
\C^*/\Q^* \ = \  (\C \otimes \Q)^*/\Q^* \ \hookrightarrow \ (\C \otimes E)^*/E^*. 
$

\medskip
\subsubsection{}
\label{sec:delta-K}
Suppose first that $K$ is a CM field. Let $K^+$ be the totally real field of which it is a quadratic extension: $K = K^+(\sqrt{D})$ with 
$D$ totally negative. As $D \in K^+$ is well-defined up to multiplication by a square, the square root in $\C$ of $N_{K^+/\Q}(D)$ 
is well-defined modulo $\Q^*.$ Define $\delta(K)$ to be its image in $\C^*/\Q^*$, as well as in $(\C \otimes E)^*/E^*.$

\medskip

In general, let $K_1$ be the largest CM field or totally real subfield of $K.$ If $K_1$ is totally real then define $\delta(K) := 1.$ 
If $K_1$ is CM, then define 
$$
\delta(K) \ := \ \delta(K_1)^{[K:K_1]}.
$$
This formula holds more generally for $K_1$ any CM field contained in $K.$ 

\begin{thm}
\label{thm:c+/c-over_total_imag}
Let $M$ be a Betti--de Rham system of realizations over a totally imaginary field $K$, with coefficients in $E$. Assume that $M$ has no middle
Hodge number. If $M$ is of dimension $n$ over $E$ (i.e., $n = \dim_E(M_\sigma)$ for any embedding $\sigma : K \to \C$), 
then one has 
$$
c^+_E(M)/c^-_E(M) \ = \ \delta(K)^n.
$$
\end{thm}

\begin{proof}
Let ${}_0M$ be deduced from $M$ by restriction of ground field from $K$ to $\Q$. By definition, 
$c_E^\pm(M) = c_E^\pm({}_0M).$ One has 
$$
{}_0M_{\rm B} \ = \ \bigoplus_{\sigma : K \to \C} M_\sigma, 
$$
and for $c$, the complex conjugation, the involution $\sF_\infty$ of ${}_0M_{\rm B}$ permutes $M_\sigma$ and $M_{c\sigma}.$
Let $\Phi$ be a CM type for $K$, i.e., it is a set of complex embeddings containing exactly one element of each pair $\{\sigma, c\sigma\}$. 
One has then:
\begin{equation}\label{eqn:0_M^+}
\begin{split}
{}_0M^+ & =  \ \bigoplus_{\sigma \in \Phi} \ (\mbox{graph of $\sF_\infty : M_\sigma \to M_{c\sigma}$}) \\
{}_0M^- & =  \ \bigoplus_{\sigma \in \Phi} \ (\mbox{graph of $-\sF_\infty : M_\sigma \to M_{c\sigma}$}), 
\end{split}
\end{equation}
and the projections to the $M_\sigma,$ for $\sigma \in \Phi$, induce an isomorphism of $E$-vector spaces:
\begin{equation}\label{eqn:0_M^+-}
{}_0M_{\rm B}^\pm \ = \ \bigoplus_{\sigma \in \Phi} M_\sigma. 
\end{equation}

\medskip

Let $p$ be the integer such that $M$ is of weight $2p-1$ or $2p-2$. The comparison isomorphism induces 
$\C \otimes E$-linear isomorphisms between the complexifications of the $E$-vector spaces ${}_0M^\pm$ and 
${}_0M_{\rm dR}/\cF^p{}_0M_{\rm dR},$ and the periods $c_E^\pm(M)$ are their determinants, computed in 
bases over $E$ of  ${}_0M^\pm$ and ${}_0M_{\rm dR}/\cF^p{}_0M_{\rm dR}.$

\medskip

The ratio $c_E^+(M)/c_E^-(M)$ is the image in $(\C \otimes E)^*/E^*$ of the determinant, 
in $(\C \otimes E)^*,$ of the automorphism of $\C \otimes (\oplus_{\sigma \in \Phi} M_\sigma),$ which
in the diagram below is the left vertical arrow, defined as the composite of the various isomorphisms or their inverses: 
\begin{equation}\label{eqn:certain_iso}
\xymatrix{
\C \otimes (\oplus_{\sigma \in \Phi} M_\sigma)  \ar[dd]  & 
\C \otimes {}_0M^+ = 
\C \otimes (\bigoplus_{\sigma \in \Phi} \ (\mbox{graph of $\sF_\infty : M_\sigma \to M_{c\sigma}$})) \ar@{>}[l] \ar@{>}[d] \\
 & \C \otimes ({}_0M_{\rm dR}/\cF^p{}_0M_{\rm dR}) \ar@{>}[d] \\
\C \otimes (\oplus_{\sigma \in \Phi} M_\sigma)  & 
\C \otimes {}_0M^- = 
\C \otimes (\bigoplus_{\sigma \in \Phi} \ (\mbox{graph of $-\sF_\infty : M_\sigma \to M_{c\sigma}$})) \ar@{>}[l] 
}
\end{equation}
As we now have an automorphism, the choices of bases are irrelevant; one should just use the same basis at the source 
and target. When we identify $\C \otimes E$ with a product of copies of $\C$ indexed by the complex embeddings 
$\tau : E \to \C$, the determinant of the isomorphism in the left vertical arrow in \eqref{eqn:certain_iso} becomes a 
family of elements of $\C$ indexed by $\tau : E \to \C.$

\medskip

Both ${}_0M_{\rm dR}$ and $\cF^p{}_0M_{\rm dR}$ are $K$-$E$-bimodules. The first is free as a $K \otimes E$-module. Indeed, 
if we choose $\sigma: K \to \C$, it becomes free after extension of scalars from $K \otimes E$ to $K \otimes \C$ by $\sigma$, 
becoming isomorphic to $\C \otimes M_\sigma$. The $K \otimes E$-module $\cF^p{}_0M_{\rm dR}$ is usually not free. 
Define $\cF^+ := \C \otimes \cF^p{}_0M_{\rm dR}.$ As a $\C$-algebra $\C \otimes K \otimes E$ is the product of copies 
of $\C$ indexed by pairs $(\sigma, \tau)$ of complex embeddings $\sigma$ of $K$ and $\tau$ of $E$; the module 
structure of $\cF^+$ amounts to a decomposition 
$$
\cF^+ \ = \ \bigoplus_{\sigma, \tau} \cF^+_{\sigma, \tau},
$$
where $\cF^+_{\sigma, \tau}$ is a finite-dimensional complex vector; we define 
$$
m(\sigma,\tau) \ := \ \dim_\C \cF^+_{\sigma, \tau}. 
$$
The set of integers $\{m(\sigma, \tau)\}_{\sigma,\tau}$ determine the isomorphism class of the $K$-$E$-bimodule $\cF^p{}_0M_{\rm dR}.$ 

\medskip

\begin{lem}
When we identify $\C \otimes E$ with a product of copies of $\C$ indexed by embeddings $\tau : E \to \C$, the determinant, in $\C \otimes E,$ of 
the automorphism of $\C \otimes (\oplus_{\sigma \in \Phi} M_\sigma),$ in \eqref{eqn:certain_iso}, is:
$$
\left(\prod_{\sigma \in \Phi} \ (-1)^{m(\sigma, \tau)}\right)_{\tau : E \to \C}.
$$
\end{lem}

\begin{proof}
We will identify $\cF^+$ with its image by the comparison isomorphism between $\C \otimes {}_0M_{\rm B}$ and $\C \otimes {}_0M_{\rm dR}.$
It is the sum of its intersections $\cF^+_\sigma$ with the summands $\C \otimes M_\sigma$, and 
$\sF_\infty(\cF^+_\sigma)$ is a supplement to $\cF^+_{c \sigma}$ in $\C \otimes M_{c \sigma}$. Indeed, $\cF^+_\sigma$ is the sum 
over $\tau$ of the $\cF_{\sigma, \tau}$, and $\sF_\infty(\cF^+)$ is the complex conjugation of $\cF^p(\C \otimes (\oplus M_\sigma)).$

\smallskip

The automorphism of $\C \otimes (\oplus_{\sigma \in \Phi} M_\sigma)$ in \eqref{eqn:certain_iso} is the product of the following automorphisms 
of the $\C \otimes M_\sigma$ ($\sigma \in \Phi$): 
\begin{equation}
\label{eqn:certain_iso_sigma}
\xymatrix{
\C \otimes M_\sigma  \ar[dd]  & 
\C \otimes (\mbox{graph of $\sF_\infty : M_\sigma \to M_{c\sigma}$}) \ar@{>}[l] \ar@{>}[d] \\
 & (\C \otimes M_\sigma)/\cF^+_\sigma + (\C \otimes M_{c\sigma})/\cF^+_{c\sigma} \ar@{>}[d] \\
\C \otimes M_\sigma  & 
\C \otimes (\mbox{graph of $-\sF_\infty : M_\sigma \to M_{c\sigma}$}) \ar@{>}[l] 
}
\end{equation}
We will conclude the proof of the lemma by showing that the determinant of the automorphism 
of $\C \otimes M_\sigma,$ the left vertical arrow in \eqref{eqn:certain_iso_sigma}, is $\{(-1)^{m(\sigma, \tau)}\}_{\tau : E \to \C}.$

\smallskip

Let us decompose $\C \otimes M_\sigma$ as the direct sum of $\cF_\sigma^+$ and $\sF_\infty(\cF^+_{c\sigma}).$ The automorphism of 
$\C \otimes M_\sigma$ in \eqref{eqn:certain_iso_sigma} is $-1$ on the first summand and $+1$ on the second. 
Indeed, as $x \in \cF^+_\sigma$ maps to $(x, \sF_\infty(x))$, equal modulo $\cF^+_\sigma$ to $(-x, \sF_\infty(x))$ in the graph of $-\sF_\infty$, 
projecting to $x$; and for $y \in \cF^+_{c\sigma}$, $\sF_\infty(y)$ maps to $\sF_\infty(y)+y$, congruent modulo $\cF^+_{c \sigma}$ to $\sF_\infty(y)-y$ projecting to $\sF_\infty(y)$. 
As $\cF^+_\sigma$ is the product of the $\cF_{\sigma, \tau}$, this gives the formula for the determinant of the automorphism 
in \eqref{eqn:certain_iso_sigma}. 
\end{proof}

\medskip
\subsubsection{}
Let $A \subset \C$ be a Galois extension of $\Q,$ large enough to contain the images of $K$ and $E$ by any complex embedding. One has 
$$
m(\sigma, \tau) \ = \ \dim_A(\cF^p {}_0M_{\rm dR} \, \otimes_{K \otimes E, \sigma \otimes \tau} \, A).
$$
It follows that the function $(\sigma, \tau) \mapsto m(\sigma,\tau)$ is $\Gal(A/\Q)$-invariant. We go through the steps in the proof of Prop.\,\ref{prop:strong-purity} one 
more time. The $\C \otimes E$-module $\C \otimes M_\sigma$ is free of rank $n$ and is the direct sum of $\cF^+_\sigma$ and 
$\sF_\infty \cF^+_{c\sigma}$. It follows that 
\begin{equation}
\label{eqn:n_sigma_n_c_sigma}
m(\sigma, \tau) + m(c\sigma, \tau) \ = \ n.
\end{equation}
For any $c' = g c g^{-1}$ conjugate in $\Gal(A/\Q)$ of the complex conjugation $c$, one has 
$$
m(\sigma, \tau) + m(c'\sigma, \tau) \ = \ n.
$$
Indeed, 
$$
m(c'\sigma, \tau) \ = \ m(g c g^{-1}\sigma, \tau) \ = \ m(c g^{-1}\sigma, g^{-1}\tau) \ = \ n - m(g^{-1}\sigma, g^{-1}\tau) \ = \ n - m(\sigma, \tau). 
$$
It follows that, as a function of $\sigma$, $m(\sigma,\tau)$ is invariant by the normal subgroup $H$ generated by $c' c''^{-1},$ for 
$c'$ and $c''$ any conjugates of $c$. The quotient of $\Gal(A/\Q)$ by $H$ is the largest quotient in which $c$ becomes central, 
and invariance of $m$ means that $m(\sigma,\tau)$ depends only, as a function of $\sigma$, on the restriction of 
$\sigma$ to the subfield $K_1$ of $K.$

\medskip

If $K_1$ is totally real, \eqref{eqn:n_sigma_n_c_sigma} shows that $m(\sigma,\tau) = n/2$, so that 
$\prod_{\sigma \in \Phi}(-1)^{m(\sigma,\tau)}$ is independent of $\tau$; it is either $+1$ or $-1$ 
for all $\tau,$ and is the image of $\pm 1 \in \E^*$ in $(\C \otimes E)^*$. As claimed in 
Thm.\,\ref{thm:c+/c-over_total_imag}, the image of $c^+_E(M)/c^-_E(M)$ in 
$(\C\otimes E)^*/E^*$ is trivial. 

\medskip

We now suppose that $K_1$ is a CM field. A subset $\Phi_1$ of the set of complex embeddings $\sigma_1$ of $K_1$ containing
exactly one element of each pair $\{\sigma_1,c\sigma_1\}$ defines a subset $\Phi$ of complex embeddings of $K$: the extensions
of elements of $\Phi_1$. Let $m_1(\sigma_1,\tau)$ be $m(\sigma,\tau)$ for any extension $\sigma$ of $\sigma_1.$ One has
$$
\prod_{\sigma \in \Phi} (-1)^{m(\sigma,\tau)} \ = \ 
\left(\prod_{\sigma_1 \in \Phi_1} (-1)^{m_1(\sigma_1,\tau)}\right)^{[K:K_1]}.
$$
To prove Thm.\,\ref{thm:c+/c-over_total_imag}, it remains to show that the image of
\begin{equation}\label{eqn:thm_reduce_K_1}
\left(\prod_{\sigma_1 \in \Phi_1} (-1)^{m_1(\sigma_1,\tau)}\right)_\tau \ \in \ (\C \otimes E)^*
\end{equation}
in $(\C \otimes E)^*/E^*$ coincides with $\delta(K_1)^n.$ This relies on the identity 
$m_1(\sigma_1,\tau) + m_1(c\sigma_1,\tau) = n,$ which implies that if we modify $\Phi_1$ to $\Phi_1'$, by replacing 
$\sigma_1 \in \Phi_1$ by $c\sigma_1$, the product \eqref{eqn:thm_reduce_K_1} gets multiplied, for each $\tau$, by $(-1)^n.$

\medskip

Let $K_1^+$ be the totally real field of which $K_1 = K_1^+(\sqrt{D})$ is a quadratic extension with $D$ a totally negative element of $K_1^+$, 
and let $d = N_{K^+_1/\Q}(D).$ We have a commutative square: 
$$
\xymatrix{
A \otimes E \ \ar@^{{(}->}[r] \ar[d]^{\simeq} & \ \C \otimes A \ar[d]^{\simeq} \\
\oplus_\tau A \ \ar@^{{(}->}[r] & \ \oplus_\tau \C.
}
$$ 
Let $d^{1/2}$ be a square root of $d$ in $A$; note that $d^{1/2} = \delta(K_1).$ 
We need to show that the ratio $(d^{1/2})^n\otimes 1 \in A \otimes E$
and \eqref{eqn:thm_reduce_K_1} in $A \otimes E$ is in $E^*,$ i.e., that the ratio is invariant by $\Gal(A/\Q).$ 

\medskip

Let $\beta$ be the following character of $\Gal(A/\Q)$: for $\Phi_1$ as before, and $g \in \Gal(A/\Q)$, 
define $\beta(g)$ to be $(-1)^k$, where $k$ is the number of 
$\sigma_1 \in \Phi_1$ which need to be replaced by $c \sigma_1$ to transform $\Phi_1$ to $g \Phi_1$.  
If we apply $g$ to the element in \eqref{eqn:thm_reduce_K_1}
$$
\tau \ \mapsto \ \prod_{\sigma_1 \in \Phi_1} (-1)^{m(\sigma_1,\tau)}
$$  
of $A \otimes E$, identified with a product of copies of $A$, one obtains
$$
g \tau \ \mapsto \ \prod_{\sigma_1 \in \Phi_1} (-1)^{m(\sigma_1,\tau)} 
\ = \ \prod_{\sigma_1 \in \Phi_1} (-1)^{m(g \sigma_1, g \tau)}, 
$$  
that is 
$$
\tau \ \mapsto \ \prod_{\sigma_1 \in g \Phi_1} (-1)^{m(\sigma_1,\tau)}. 
$$  
This is simply \eqref{eqn:thm_reduce_K_1} multiplied by $\beta(g)^n.$ On the other hand, a square root of $d$ in $A$ is obtained by taking 
$$
d^{1/2} \ = \ \prod_{\sigma \in \Phi_1} \sigma(\sqrt{D}). 
$$
For this square root we have
$
g(d^{1/2}) \ = \ \beta(g) d^{1/2}.$
Hence, the ratio of $(d^{1/2})^n\otimes 1$ and \eqref{eqn:thm_reduce_K_1} in $A \otimes E,$ being invariant by $\Gal(A/\Q),$ is in $E^*.$ 
This completes the proof of Thm.\,\ref{thm:c+/c-over_total_imag}. 
\end{proof}

The reader should compare the above theorem with some of the motivic calculations in Harris \cite{harris-crelle} and Hida \cite{hida-duke}. 

\medskip
\subsubsection{Rankin--Selberg $L$-functions over a totally imaginary field}
\label{sec:tot-imag-field-comparison}
The period relation in Thm.\,\ref{thm:c+/c-over_total_imag} has a direct bearing on the rationality results in \cite{raghuram-imaginary} for the 
Rankin--Selberg $L$-functions for $\GL(n) \times \GL(n')$ over a totally imaginary field $F$. Suppose $\sigma$ (resp., $\sigma'$) is a cohomological cuspidal automorphic 
representation of $\GL(n)$ (resp., $\GL(n')$) over $F$, and $E$ is a Galois extension of $\Q$ inside $\bQ$ which contains a copy of $F$ and the Hecke eigenvalues 
of $\sigma$ and $\sigma'$ then the main theorem of that article \cite[Thm.\,5.16]{raghuram-imaginary} states that when we have two successive critical points, say $m$ and $m+1$, then 
$$
|\delta_{F/\Q}|^{- \tfrac{n n'}{2}} \cdot \frac{L(m, \sigma \times \sigma')}{L(m+1, \sigma \times \sigma')} \in E, 
$$
where $L(s, \sigma \times \sigma')$ is the completed Rankin--Selberg $L$-function, and $\delta_{F/\Q}$ is the absolute discriminant of $F/\Q$. 
Furthermore, for all $\gamma \in \Gal(\bQ/\Q)$ we have the reciprocity law
$$
\gamma\left(|\delta_{F/\Q}|^{- \tfrac{n n'}{2}} \cdot \frac{L(m, \sigma \times \sigma')}{L(m+1, \sigma \times \sigma')}\right) \ = \ 
 \varepsilon(\gamma) \cdot \varepsilon'(\gamma) \cdot 
|\delta_{F/\Q}|^{- \tfrac{n n'}{2}} \cdot \frac{L(m, {}^\gamma\sigma \times {}^\gamma\sigma')}{L(m+1, {}^\gamma\sigma \times {}^\gamma\sigma')}.  
$$ 
where $ \varepsilon(\gamma), \varepsilon'(\gamma) \in \{\pm 1\}$ are certain signatures that appear due to Galois action on cohomology of the boundary of the Borel--Serre compactification of a 
locally symmetric space for $\GL(N)$, with $N = n+n'.$ If we take $M$ to be the tensor product over $E$ of the motives $M(\sigma)$ and $M(\sigma')$ 
that conjecturally correspond to $\sigma$ and $\sigma'$, then Thm.\,\ref{thm:c+/c-over_total_imag} applied to this $M$ is compatible with the above mentioned 
\cite[Thm.\,5.16]{raghuram-imaginary} after the following equality of signatures 
$$
\frac{ \gamma(|\delta_{F/\Q}|^{n n'/2})} {|\delta_{F/\Q}|^{n n'/2}} \cdot 
\varepsilon(\gamma) \cdot \varepsilon'(\gamma) \ = \ 
\frac{ \gamma((i^{[F:\Q]/2} \, \delta(F))^{nn'})}{(i^{[F:\Q]/2} \, \delta(F))^{nn'}} 
$$
that is proved in \cite[Prop.\,5.26]{raghuram-imaginary}. It is interesting to reflect on the idea that the two sides of the above equality of signatures have very different origins: 
the signature on the left hand side comes from 
$|\delta_{F/\Q}|^{n n'/2}$ (that is used to normalize a measure--\cite[5.2.3.5]{raghuram-imaginary}) and $\varepsilon(\gamma), \ \varepsilon'(\gamma)$ which arise from Galois action in 
boundary-cohomology, or  what amounts to the same, on Eisenstein-cohomology; the signature on the right hand side is intrinsic to the base field as in the proof of 
Thm.\,\ref{thm:c+/c-over_total_imag}  above.

\medskip
\subsection{Tensor product motives over a number field.}
\label{sec:tensor-product-motives-general}

In this section, we will extend the results of Sect.\,\ref{sec:tensor-product-motives} to the case of systems of realizations over a totally real number field, at least 
in the case of vanishing middle Hodge numbers. 

Let $M'$ and $M''$ be Betti--de Rham systems of realizations over a number field $L$ (\ref{sec:realizations_over_number_field_L}). A tensor product 
is defined in the obvious way. It does not commute with restriction from $L$ to $\Q$. To interpret $\Res_{L/\Q}(M' \otimes M'')$, in terms of 
which $c^+$ and $c^-$ of $M' \otimes M''$ are defined, one should view $\Res_{L/\Q}(M')$ and $\Res_{L/\Q}(M'')$ as modules over $H^0(\Spec \, L)$ 
(\ref{sec:realizations_over_L_modules}). As $H^0(\Spec \, L)$-modules, 
\begin{equation}
\label{eqn:tensor_res_over _L}
\Res_{L/\Q}(M' \otimes M'') \ = \ \Res_{L/\Q}(M') \otimes_{H^0(\Spec \, L)} \Res_{L/\Q}(M''). 
\end{equation}
The same applies for systems of realizations with coefficients in $E$, and tensor product of $M'$ and $M''$ over $E$; in \eqref{eqn:tensor_res_over _L}, 
$H^0(\Spec \, L)$ should be replaced by $H^0(\Spec \, L) \otimes E.$

\begin{prop}
\label{prop:tensor_L_E}
Let $M'$ and $M''$ be systems of Betti--de Rham realizations over a totally real number field $L$, with coefficients in $E$. If 
$M'$, $M'',$ and $M' \otimes_E M''$ have vanishing middle Hodge numbers, then as elements of $(\C \otimes E)^*/E^*$ we have
$$
c^+(M' \otimes_E M'') \ = \ c^-(M' \otimes_E M''). 
$$
\end{prop}

\begin{proof}
We will consider $\Res_{L/\Q}(M' \otimes M'')$ as derived by the multilinear algebra construction \eqref{eqn:tensor_res_over _L} 
from $\Res_{L/\Q}(M')$, $\Res_{L/\Q}(M'')$, and $H^0(\Spec \, L)$. The structures used are the ring structure of $H^0(\Spec \, L),$ and 
the $H^0(\Spec \, L)-$module and $E-$module structures of $\Res_{L/\Q}(M')$ and $\Res_{L/\Q}(M'')$. 

The Betti realization of $H^0(\Spec \, L)$, with its ring structure $\fS_L$, is the sum of copies of $\Q$ indexed by the set $\Sigma_L := \Hom(L, \C)$ of the 
complex embeddings of $L$. The corresponding group $G(\fS_L)_B$ is the symmetric group $S_{\Sigma_L}$ of permutations of $\Sigma_L.$ 
Using more multilinear algebra structures on $H^0(\Spec \, L)$, one could reduce it to the Galois group of the normal closure 
of $L$ in $\C$, but this will not be needed. 

For $\Res_{L/\Q}(M')$ with its $H^0(\Spec \, L)-$module and $E-$module structures $\fS'$, the Betti realization is  
$\bigoplus_{\sigma \in \Sigma_L} M'_\sigma$, and the group $G(\fS')_B$ is an extension of $S_{\Sigma_L}$ by the product of the $\GL_E(M'_\sigma).$ The real Frobenius $\sF_\infty$ is in this product group. 
One has 
$$
\Res_{L/\Q}(M')^\pm \ = \ \bigoplus_{\sigma \in \Sigma_L} M'_\sigma{}^\pm, 
$$
and for each $\sigma$, $\dim(M'_\sigma{}^\pm) = \dim(M'_\sigma)/2.$ It follows that the centralizer $Z(\sF_\infty)$ maps onto $S_{\Sigma_L}$. 

For $\Res_{L/\Q}(M'),$  $\Res_{L/\Q}(M'')$ and $H^0(\Spec \, L)$, and the ring and module structures $\fS$, one has similarly, 
$$
G(\fS)_B \ = \ 
S_{\Sigma_L} \ltimes \left(\prod_\sigma \GL_E(M'_\sigma) \times \prod_\sigma \GL_E(M''_\sigma) \right).
$$

We need to prove that the determinants of the actions of $Z(\sF_\infty)$ on the $+$ and $-$ parts of $\Res_{L/\Q}(M' \otimes M'')$ are the same. 
For the connected component 
$$
\prod_{\sigma \in \Sigma_L} 
\left(
\GL_E(M'_\sigma{}^+) \times 
\GL_E(M'_\sigma{}^-) \times
\GL_E(M''_\sigma{}^+) \times
\GL_E(M''_\sigma{}^-) 
\right)
$$
the proof is the same as in Sect.\,\ref{sec:tensor-product-motives}. Let us pick bases $e'^+_{i, \sigma}$, $e'^-_{i, \sigma}$, $e''^+_{j, \sigma}$, and $e''^-_{j, \sigma}$, 
where $1 \leq i \leq \dim(M')/2$ and $1 \leq j \leq \dim(M')/2$, of the spaces $M'_\sigma{}^{\pm}$ and $M''_\sigma{}^{\pm}$. We can lift 
$S_{\Sigma_L}$ in $G(\fS)_B$ so that for each decoration $'$, $''$, ${}^+$, ${}^-$, $i$ or $j$, it permutes the $e_\sigma.$ For this choice, the action of 
$S_{\Sigma_L}$ on $\Res_{L/\Q}(M' \otimes M'')^\pm$ is the sum of $\dim(M' \otimes M'')/2$ permutation representation: the same for $+$ and $-$, hence the 
same determinant. 
\end{proof}

\medskip
\subsection{Motives via multiplication induction: Asai motives.}

\subsubsection{Definition}
Consider an extension $K/F$ of number fields, and a Betti--de Rham realization $M$ over $K$ of rank $n$ and weight $w$. 
There is a notion of {\it multiplicative induction from $K$ to $F$ of $M$}: 
$$
M \ \mapsto \ \bigotimes_{K/F} M, 
$$
which is a Betti--de Rham realization over $F$ of rank $n^{[K:F]}$ and weight $[K:F] w$. To define $\otimes_{K/F} M$ over $F$, it suffices to define it locally 
for the etale topology over $\Spec(F).$ This means defining what should be its base-change to $\Spec(F')$, for $F'$ a large enough finite extension of $F$, defining a 
compatibility isomorphism for an iterated extension $F''/F'/F$, and a compatibility between them for $F'''/F''/F'/F$. Here, {\it large enough} means that it contains a copy of the 
normal closure of $K/F$. For $F'$ large enough, there are $[K:F]$ embeddings over $F$
of $K$ in $F'$, and the base-change to $F'$ of $\bigotimes_{K/F} M$ is the tensor product of the corresponding base-changes of $M$. 
More generally, one can take for $K$ an etale algebra over $F$, that is a finite product of finite extensions $K_i$ of $F$. A Betti--de Rham realization $M$ over $K$ is a family 
$M_i$ of realizations over $K_i$. The definition above continues to make sense, and 
$$
\bigotimes_{K/F} M \ = \ \bigotimes_i \bigotimes_{K_i/F} M_i
$$
Statements true for tensor products can often be generalized to multiplicative induction. 

\begin{thm}
\label{thm:asai-motives}
Let $K/F$ be an extension of totally real number fields of degree $[K:F] > 1.$  Let $M$ be a Betti--de Rham realization over $K$ of rank $n$ with coefficients in a number field $E$. 
Assume that $M$ and $\otimes_{K/F} M$ have vanishing middle Hodge numbers. Necessarily, $n$ is even. 
Then as elements of $(E \otimes \C)^*/E^*$ one has
$$
c^+(\Res_{F/\Q}(\otimes_{K/F} M)) = c^-(\Res_{F/\Q}(\otimes_{K/F} M))
$$
unless $[K:F] = 2$ and $n/2$ is odd, in which case if $K = F(\sqrt{D})$ then 
$$
c^+(\Res_{F/\Q}(\otimes_{K/F} M))/c^-(\Res_{F/\Q}(\otimes_{K/F} M)) \ = \ \sqrt{N_{F/\Q}(D)},
$$
where $\sqrt{N_{F/\Q}(D)}$ is viewed as an element of $\C^*/\Q^* \hookrightarrow (\C \otimes E)^*/E^*.$
\end{thm}

\begin{proof}

The method of proof is the same as in Sect.\,\ref{sec:tensor-product-motives-general}, starting this time with $H^0(\Spec(F))$ and $H^0(\Spec(K)),$ their ring structures, and the 
morphism 
$$
H^0(\Spec(F)) \ \to \ H^0(\Spec(K)), 
$$
$M$ and its $H^0(\Spec(K))-$module and $E$-module structures.

Let $\Sigma_K = \Hom(K,\C) = \Hom(K,\R)$ be the set of all complex (real) embeddings of $K$, and similarly $\Sigma_F$. Restriction to $F$ gives the surjective map 
$$
\xymatrix{
J := \Sigma_K \ar[d]^{\pi} \\
I := \Sigma_F. 
}
$$
If we consider just $H^0(\Spec(F))$ and $H^0(\Spec(K)),$ their ring structures, and the 
morphism $H^0(\Spec(F)) \to H^0(\Spec(K)),$ one has $G(\fS)_B = \Aut(J \to I)$, the group of all permutations of $J$ compatible with $\pi : J \to I.$
All the fibres of $\pi$ have cardinality $[K:F]$, and $\Aut(J \to I)$ is a wreath product: an extension
$$
1 \ \longrightarrow \ 
\prod_{i \in I} S_{\pi^{-1}(i)}  \ \longrightarrow \ 
\Aut(J \to I) \ \longrightarrow \ 
S_I \ \longrightarrow \ 
1.
$$
If we consider, in addition, $M$ and its module structures, one has for the Betti realizations: 
$$
\Res_{K/\Q}(M)_{\rm B} \ = \ \bigoplus_{j \in J} M_j, \quad {\rm and} \quad 
\Res_{F/\Q}(\otimes_{K/F} M)_{\rm B} \ = \ \bigoplus_{i \in I} \bigotimes_{\pi(j) = i} M_j.
$$
and $G(\fS)_B$ is an extension
$$
1 \ \longrightarrow \ 
\prod_{j \in J} \GL_E(M_j) \ \longrightarrow \ 
G(\fS)_B \ \longrightarrow \ 
\Aut(J \to I) \ \longrightarrow \ 
1.
$$
This is as in Sect.\,\ref{sec:tensor-product-motives-general}, except that we have used an additional structure on $H^0(\Spec(K))$ to reduce $S_J$ to $\Aut(J \to I).$ 
Here $\sF_\infty$ is the family of involutions in $\prod_{j \in J} \GL(M_j)$, defining decompositions 
$M_j = M_j^+ \oplus M_j^-$. Since $M$ has vanishing middle Hodge number, one has $\dim(M_j^+) = \dim(M_j^-) = \dim(M)/2 = m$ (say). The equality of the dimensions 
of all the $M_j^+$, $j \in J$, ensures that $Z(\sF_\infty)$ maps onto $\Aut(J \to I)$.

\medskip

We need to compare the characters of $Z(\sF_\infty)$ acting on 
$$
\Res_{F/\Q}(\otimes_{K/F} M)^+ \ = \ 
\ \bigoplus_{i \in I} \bigotimes_{\substack{\varepsilon : \pi^{-1}(i) \to \{0,1\} \\ \sum_{j \in \pi^{-1}(i) } \varepsilon(j) = 0}} M_j^{\varepsilon(j)}, 
$$
$$
\Res_{F/\Q}(\otimes_{K/F} M)^- \ = \ 
\ \bigoplus_{i \in I} \bigotimes_{\substack{\varepsilon : \pi^{-1}(i) \to \{0,1\} \\ \sum_{j \in \pi^{-1}(i) } \varepsilon(j) = 1}} M_j^{\varepsilon(j)}, 
$$
where, $M_j^0 = M_j^+$ and $M_j^1 = M_j^-.$

\medskip


For the action of the connected component $Z(\sF_\infty) = \prod \GL(M_j^+) \times \GL(M_j^-)$, only the $\pi(j)$-summand matters for the 
determinant of the action of $\GL(M_j^+)$ (resp., $\GL(M_j^-)$). The action is on 
$$
M_j^+ \otimes ({\rm something})
$$
with $\dim({\rm something})$ being the same for $+$ or $-$. The same for $\GL(M_j^-)$. It remains to see what happens on the representatives of the connected components. 

\medskip
Fix bases $\{e_{j,1}^\pm, \dots, e_{j,m}^\pm\}$ for $M_j^\pm.$ As a representative for $\sigma$, we take $\{e_{j,1}^\pm, \dots, e_{j,m}^\pm\} \to \{e_{\sigma(j),1}^\pm, \dots, e_{\sigma(j),m}^\pm\}$. 
We need to compare the actions of $\det(\sigma)$ on $\Res_{F/\Q}(\otimes_{K/F} M)^\pm.$

\medskip
Action of $S_{\pi^{-1}(i)}$, the permutation group of the fibre above $i$: 
It suffices to compare the actions on $\otimes_{\pi(j) = i} M_j^{\varepsilon(j)}.$ For a transposition of two $j$'s; call them $1$ and $2$; if $[K:F] > 2$, we have---for some 
$W^\pm$ of the same dimension---to consider 
$$
W^+ \otimes \left((M_1^+ \otimes M_2^+) \oplus (M_1^- \otimes M_2^-) \right) \ \oplus \ 
W^- \otimes \left((M_1^+ \otimes M_2^-) \oplus (M_1^- \otimes M_2^+) \right), 
$$
and 
$$
W^- \otimes \left((M_1^+ \otimes M_2^+) \oplus (M_1^- \otimes M_2^-) \right) \ \oplus \ 
W^+ \otimes \left((M_1^+ \otimes M_2^-) \oplus (M_1^- \otimes M_2^+) \right). 
$$
The determinant is the same since $\dim(W^+) = \dim(W^-).$ Only the case $[K:F] = 2$ remains to be considered. Here, we have to compare 
$$
(M_1^+ \otimes M_2^+) \oplus (M_1^- \otimes M_2^-) \quad {\rm and} \quad (M_1^+ \otimes M_2^-) \oplus (M_1^- \otimes M_2^+).
$$
For the latter, we have an involution with no fixed points on the basis vectors, the determinant of the involution is $(-1)^{m^2};$  
on the former, we have $\dim(M)$ many fixed-points: 
$e_{1,a}^+ \otimes e_{2,a}^+$ and $e_{1,a}^- \otimes e_{2,a}^-$, hence the determinant of the involution is $(-1)^{m^2-m}(-1)^{m^2-m} = 1;$ whence, the ratio 
of the two determinants is $(-1)^{m^2} = (-1)^m = (-1)^{\dim(M)/2}.$  For $\dim(M)/2$ even, we get the same trivial determinant on $\prod_{i \in I} S_{\pi^{-1}(i)}$, and 
if $\dim(M)/2$ is odd then we get $\prod_{i \in I} {\rm sgn}.$

\medskip
Let us choose an isomorphism between $J \to I$ and $(J_0 \times I) \to I$ so that 
$$
\Aut(J \to I) \ = \ S_I \times \prod_{i \in I} S_{J_0}.
$$
Except in the case of $[K:F] = 2$, and $\dim(M)/2$ is odd, we have an equality of characters on $\prod_{i \in I} S_{J_0}.$ We claim that in this case we have an equality of characters 
on all of $\Aut(J \to I)$ (and hence $c^+ \approx c^-$). Indeed, the $(\otimes_{\pi(j) = i} M_i)^+$ for all $i$ have been identified, and this gives the same action. Same for $-$, and we have 
the same dimension, hence the same character. 

\medskip
The same applies to the case $[K:F] = 2$ and $\dim(M)/2$ odd, in which case the ratio of the two characters is 
$$
\begin{cases} \mbox{$1$ on $S_I$,}\\ \mbox{$\prod_{i \in I} {\rm sgn}$ on $\prod_{i \in I} S_{\pi^{-1}(i)}$} \end{cases} 
$$
This character does not depend on how the $\pi^{-1}(i)$ have all been identified with a fixed $J_0$; two such identifications are conjugate by an element of 
$\prod_{i \in I} S_{\pi^{-1}(i)}$. We have inside $\Aut(J \to I)$, the Galois group of the normal closure of $F/\Q.$ Restricting the character to this Galois group 
gives us a quadratic extension $\Q$; if $K = F(\sqrt{D})$, then this quadratic extension is $\Q(\sqrt{N_{F/\Q}(D)}).$ It gives an Artin motive $A$ of rank $1$ over $\Q$, 
and in the case $[K:F] = 2$ and $\dim(M)/2$ odd, we get
$$
c^+(\Res_{F/\Q}(\otimes_{K/F} M)) \ \approx \ c^-(\Res_{F/\Q}(\otimes_{K/F} M)) c^+(A).
$$

\end{proof}

\subsubsection{Motivating example-I}
Consider a projective nonsingular variety $X$ over $K$, and the variety $\prod_{K/F}X$ over $F$ given by Weil restriction of scalars. For any $F$-algebra $A$, 
one defines $(\prod_{K/F}X)(A) = X(K \otimes_F A).$ The motivating example for multiplicative induction comes from relating the cohomology of 
$\prod_{K/F}X$ and the cohomology of $X$. For example, for the Betti realizations, take $A = \C$, 
fix an embedding $\sigma : F \to \C$, and fix an ordering 
$\{\eta_1,\dots,\eta_d\}$ on the set of embeddings $K \to \C$ that restrict to $\sigma$, where $d = [K:F]$. For the $\C$-points, one has: 
$$
(\prod_{K/F}X)_\sigma (\C) \ = \ \prod_{j = 1}^d X_{\eta_j}(\C), 
$$
where, 
$(\prod_{K/F}X)_\sigma = (\prod_{K/F}X) \times_{F, \sigma} \C$ and $X_{\eta_j} := X \times_{K, \eta_j} \C.$
By the K\"unneth theorem, one has
\begin{equation}
\label{eqn:kunneth-asai}
H^p((\prod_{K/F}X)_\sigma (\C), \Q) \ = \ \bigoplus_{\sum p_j = p} \bigotimes_{j=1}^d H^{p_j}(X_{\eta_j}(\C), \Q). 
\end{equation}
Now consider the Betti--de Rham realization for the motive $H^q(X)$ over $K$ with coefficients in $\Q$. If $q$ is even, 
the Betti realization of the multiplicative induction
$\bigotimes_{K/F} H^q(X)$ is the piece in $H^{q[K:F]}((\prod_{K/F}X) (\C), \Q)$ corresponding to the summand in \eqref{eqn:kunneth-asai} with $p_1 = \cdots = p_d = q.$

\subsubsection{Motivating example-II: Asai motives}
To a Hilbert modular form $f$ for a real quadratic extension $K/\Q$, Asai \cite{asai} attached a degree-$4$ $L$-function which has all the usual analytic properties of an automorphic $L$-function 
under suitable assumptions on $f$. 
This $L$-function is now called the Asai $L$-function, or the twisted tensor $L$-function. 
As is now well-known, the Asai $L$-function is the standard $L$-function of a 
suitable automorphic form ${\rm As}(f)$ on $\GL(4)/\Q$ obtained by Langlands's principle of functoriality which on the Galois side is given by multiplicative induction: a two-dimensional irreducible representation $V$ of $\Gal(\bQ/K)$ corresponding to $f$ gives after multiplicative induction a four-dimensional representation $\As(V) = \otimes_{K/\Q} V$ of $\Gal(\bQ/\Q)$ that corresponds to ${\rm As}(f)$; see \cite{muthu}. 
The Asai $L$-function is one of the $L$-functions in the Langlands-Shahidi family of $L$-functions. 
More generally, given a quadratic extension of totally real number fields $K/F$, and given a cuspidal automorphic representation $\pi$ of $\GL_n/K$, the principle of functoriality predicts an automorphic representation ${\rm As}(\pi)$ of 
$\GL_{n^2}/F$, which is cuspidal under some regularity assumption on $\pi$. Furthermore, suppose that $\pi$ is motivic and $M = M(\pi)$ is the conjectural motive associated to $\pi$, then 
conjecturally, ${\rm As}(\pi)$ is also motivic, and one expects $\otimes_{K/\Q}(M(\pi)) \ = \ M({\rm As}(\pi)).$
The period relation in Thm.\,\ref{thm:asai-motives} then gives a motivic explanation to the rationality results for 
the ratios of critical values of Asai $L$-functions in the forthcoming \cite{muthu-raghuram}.

\medskip
\subsection{Symplectic and Orthogonal Motives}

\subsubsection{The case when the base field is $\Q$}
We consider a Betti--de Rham system of realizations $M$ over $\Q$ with coefficients in $E$. 
In the even weight case we assume the vanishing of the middle Hodge number. We assume $M$ is given a linear algebra structure $\fS$ such that 
$G(\fS)$ is contained in a group $\GSp$ of symplectic similitudes, or in the connected component $\GO^\circ$ of a group $\GO$ of orthogonal similitudes. 
The group $\GO^\circ$ is also sometimes denoted $\GSO.$

\medskip
\begin{thm}
\label{thm:gsp-go}
If the multiplier of $\sF_\infty$ is $1$, then $c^+_E(M) = c^-_E(M)$.
\end{thm}

\medskip


Let us elaborate our setting a bit more. One starts with $M$ of rank $2n$ over $E,$ and with an $L$ of rank $1$ over $E$. In practice (motivic case), $L$ will be $L(\chi)(m),$ 
with $L(\chi)$ the system of realizations defined by a Dirichlet character $\chi$ with values in $E$ and with $m \in \Z.$ The assumption of 
`the multiplier of $\sF_\infty$ is $1$' translates to $\chi(-1) = (-1)^m.$ 
In the symplectic case, the linear algebra structure $\fS$ is a non-degenerate alternating $E$-bilinear pairing 
$$
B  : M \otimes_E M \to L.
$$
In the orthogonal case, one has to give more than a non-degenerate symmetric $E$-bilinear pairing 
$$
B  : M \otimes_E M \to L.
$$
Such a pairing induces a similar pairing on $\det(M) := \wedge^{2n} M$: 
$$
\det(B) : \det(M)^{\otimes^2} \to L^{\otimes^{2n}}.
$$
This new pairing amounts to a triviliazation of 
$$
(\det(M) \otimes L^{\otimes^{-n}})^{\otimes^2}.
$$
The required multilinear algebra structure $\fS$ consists of the $B$, together with a trivialization of $\det(M) \otimes L^{\otimes^{-n}}$ (isomorphism with the unit object) 
inducing the above trivialization of its tensor square.

\medskip
\begin{proof}
For $x \in M^+$ and $y \in M^-$, on the one hand, one has
$B(\sF_\infty x, \sF_\infty y) = B(x, -y) = -B(x,y),$ but on the other hand one also has
$B(\sF_\infty x, \sF_\infty y) = B(x, y),$
hence, $B(x,y) = 0;$ whence, $M_B$ is the symplectic or orthogonal direct sum of $M^+$ and $M^-.$  

\smallskip

We have to show that $Z(\sF_\infty)$ acts with the same determinant on $M^+$ and $M^-$. The dimension $n$ of 
$M^+$ and $M^-$ is even. In the symplectic case, because $M^+$ and $M^-$ are symplectic. In the 
orthogonal case, the connected component of the group of similitudes is given by 
$$
\det(g) = {\rm multiplier}(g)^n,
$$
forcing $n$ to be even since the multiplier of $\sF_\infty$ is $1$ and determinant $(-1)^n.$ 

\smallskip

An element $g$ in $G(\fS)$ centralizing $\sF_\infty$ is the sum of an $g^+$ acting on $M^+$ and $g^-$ acting on $M^-$. In the symplectic case, $g^+$ and $g^-$ are symplectic 
similitudes with the same multipliers, hence the same determinant
$$
\det(g^\pm) = {\rm multiplier}(g^\pm)^{n/2}.
$$

\smallskip

In the orthogonal case, they are orthogonal similitudes with the same multiplier. The ratios 
$$
\det(g^\pm)/{\rm multiplier}(g^\pm)^{n/2} \in \{\pm 1\}, 
$$
but $\det(g)/{\rm multiplier}(g)^{n} = 1$ means that they are equal: for some sign $\epsilon$, one has
$$
\det(g^\pm) = {\rm multiplier}(g^\pm)^{n/2}\cdot\epsilon,
$$
proving that $\det(g^+) = \det(g^-).$
\end{proof}

\medskip
\subsubsection{The case when the base field is totally real}
\label{sec:symp-orth-tot-real}
Thm.\,\ref{thm:gsp-go} will continue to hold for $M$ and $L$ Betti--de Rham systems of realizations over a totally real field $F$. The parity assumption is now that for 
each complex embedding $\sigma$ of $F$, the involution $\sF_\infty$ of $L_\sigma$ is trivial. In the motivic case, one will want $L$ to be $L(\chi)(m)$ for $\chi$ an 
abelian character of $\Gal(\overline{F}/F)$ with values in $E^*$, and the parity assumption is that for each infinite Frobenius $\sF_\infty$ one has 
$\chi(\sF_\infty) = (-1)^m.$
The relevant group $G(\fS)$ for $\Res_{F/\Q}(M)$ is now an extension of the symmetric group $S_T,$ of the set $T$ of complex embeddings of $F$, by 
the product of the (connected components) of the groups of similitudes of the $M_\sigma$'s. The group $Z(\sF_\infty)$ is the extension of $S_T$ by the product 
of the centralizers. The proof that the action of this product of centralizers on 
$$
\Res_{F/\Q}(M)^+ = \prod M_\sigma^+, \quad {\rm and} \quad \Res_{F/\Q}(M)^- = \prod M_\sigma^-, 
$$
have the same determinant is as before. It remains to compute a ratio of determinants on the quotient $S_T$. For this, we may first extend the field of coefficients to 
split the forms $B$ on $M_\sigma$. One can then find basis $e^\pm_{i, \sigma}$ of the $M_\sigma^\pm$ and representatives for elements of $S_T$ 
which for each $(\pm, i)$, permute the $e^\pm_{i, \sigma}$. As $n$ is even, for such a choice both determinants are trivial.

\medskip
\subsubsection{Example: The tensor product of two rank-two motives-II}
\label{sec:prod-rank-two-motives-2}
We revisit the example discussed in \ref{sec:prod-rank-two-motives-1}. Using the notations therein, if $\varphi \in S_k(\Gamma_0(N), \omega)$ is a primitive cuspform and 
$M(\varphi)$ the associated rank-two motive, then $M(\varphi)$ is symplectic but does not satisfy the parity condition in the hypothesis of Thm.\,\ref{thm:gsp-go}. The symplectic
form $B$ is an alternating form $B_\varphi : M(\varphi) \otimes M(\varphi) \to \Q(\omega)(1-k).$ 
Since the Dirichlet character $\omega$ is 
the nebentypus character of a weight $k$ modular form, one has $\omega(-1) = (-1)^k$. Hence $\sF_\infty$ acts on the rank-one motive $\Q(\omega)(1-k)$ via 
$\omega(-1)\cdot(-1)^{1-k} = -1.$ In particular, Thm.\,\ref{thm:gsp-go} does not apply to $M(\varphi)$--but this is expected since the two periods for a general (non-CM) cuspform 
$\varphi \in S_k(\Gamma_0(N), \omega)$ are expected to be algebraically independent. 

\smallskip

Now suppose, for $j = 1,2$ we have $\varphi_j \in S_{k_j}(\Gamma_0(N_j), \omega_j).$ Assume $k_1 \neq k_2.$ The orthogonal form 
$B_{\varphi_1} \otimes B_{\varphi_2}$ on the tensor product motive $M(\varphi_1) \otimes M(\varphi_2)$ takes values in the rank-one motive $\Q(\omega_1\omega_2)(2-k_1-k_2)$, on which 
the action of $\sF_\infty$ is via $(\omega_1\omega_2)(-1)\cdot(-1)^{2-k_1-k_2}  = 1,$ satisfying the hypothesis of the theorem. It remains to check that $G(\fS)$ is contained in the connected 
component $\GO^\circ$. This is a special case of the following general lemma.

\begin{lem}
Let $(V_1, B_1)$ and $(V_2,B_2)$ be nondegenerate symplectic vector spaces over a field $F$. Then $V = V_1\otimes_F V_2$ equipped with $B = B_1 \otimes B_2$ is a nondegenerate 
orthogonal space over $F$. Under the canonical map $\GL(V_1) \times \GL(V_2) \to \GL(V_1 \otimes V_2)$, the image of 
$\GSp(V_1,B_1) \times \GSp(V_2,B_2)$ is contained inside $\GSO(V, B).$  (In particular, $\GL(2) \times \GL(2)$ maps inside $\GSO(4).$)
\end{lem}

\begin{proof}
Suppose the dimension of $V_1$ is $2n_1$ and 
$\GSp(V_1,B_1)$ has a similitude character $\mu_1$ then for $g_1 \in \GSp(V_1,B_1)$ one has $\det(g_1)^2 = \mu_1(g_1)^{2n_1}.$ Similarly, 
for $g_2 \in \GSp(V_2,B_2),$ one has $\det(g_2)^2 = \mu_2(g_2)^{2n_2}.$ For the element $g_1\otimes g_2$ acting on $(V,B)$, 
the similitude character $\mu$ and the determinant are given by
$\mu(g_1\otimes g_2) = \mu_1(g_1)\mu_2(g_2)$ and $\det(g_1\otimes g_2) = \det(g_1)^{2n_2} \det(g_2)^{2n_1},$ respectively. 
One checks that $\det(g_1 \otimes g_2) \cdot \mu(g_1\otimes g_2)^{\dim(V_1\otimes V_2)/2} = 1$ which implies that $g_1 \otimes g_2 \in \GSO(V, B).$
\end{proof}

\medskip
\subsubsection{Example: The conjectural motive attached to representations of $\OO(2n)$}
In this paragraph we discuss the relation of Thm.\,\ref{thm:gsp-go} with the main result of \cite{bhagwat-raghuram}. Suppose $\OO(2n)$ is a split orthogonal group
over $\Q$, and $\sigma$ is a cuspidal automorphic representation 
of $\OO(2n)(\A_\Q)$ satisfying the conditions: 
\begin{enumerate}
\item[(i)] the Arthur parameter $\Psi$ of $\sigma$ is a cuspidal automorphic representation of $\GL(2n)/\Q$;
\item[(ii)] $\sigma$ is globally generic with respect to some Whittaker datum $\psi$;
\item[(iii)] $n$ is even, and the $\sigma_\infty$ is a discrete series representation of $\OO(n,n)$ which is locally generic with respect to $\psi_\infty.$
\end{enumerate} 
We will see below that Thm.\,\ref{thm:gsp-go} applies to the conjectural motive $M$ attached to $\sigma$ giving us then a motivic explanation
of the main theorem of \cite{bhagwat-raghuram} recalled in \eqref{eqn:BR}.

\smallskip

The motive is attached not so much as to $\sigma$, but rather to its Arthur parameter $\Psi$, which by (i) is a cuspidal automorphic representation of $\GL(2n)/\Q$. 
However, $\Psi$ is not a representation that is of motivic type, in as much as it is not algebraic, and as in \cite{clozel}, one remedies the situation by considering: 
$$
\Psi^T := \Psi \otimes |\!|\ |\!|^{1/2}
$$
a half-integral Tate-twist of $\Psi$ (one may also consider $\Psi \otimes |\!|\ |\!|^{m/2}$ for any odd integer $m$) which is of motivic type. Let $M = M(\Psi^T)$ be the conjectural rank $2n$ motive associated to $\Psi^T.$ For the expected details of this dictionary between motives and automorphic forms, the reader is referred to \cite[Chap.\,7]{harder-raghuram-book}. 

\smallskip

To get the Hodge types of $M$, recall from \cite[2.2.8]{bhagwat-raghuram} that the Langlands parameter of the discrete series representation $\sigma_\infty$, which by definition of 
the transfer $\sigma \mapsto \Psi$ is the same as the Langlands parameter $\varphi(\Psi_\infty)$ of $\Psi_\infty,$ is of the form:
$\oplus_{j=1}^nI(\chi_{\ell_j})$, for a decreasing sequence of even positive integers $\ell_1 > \ell_2 > \cdots > \ell_n$, 
where $\chi_\ell$ is the character of $\C^*$ that maps $z$ to $(z/\bar{z})^{\ell/2}$, and $I(\chi_\ell)$ is its induction to a $2$-dimensional irreducible 
representation of the Weil group $W_\R$ of $\R$. We can read off the exponents of the characters of $\C^*$ appearing in $\varphi(\Psi^T_\infty)$. One aspect of the dictionary 
$\Psi^T \leftrightarrow M$ is that if $(p,q)$ is a Hodge type for $M$ if and only if $z^{-p}\bar{z}^{-q}$ appears in $\varphi(\Psi^T_\infty \otimes |\!| \ |\!|^{(1-2n)/2}$). One concludes that 
the Hodge types of $M(\Psi^T)$ are of the form:
$$
\left\{ \left( \frac{\pm \ell_j+2n-2}{2}, \frac{\mp\ell_j+2n-2}{2} \right) \right\}_{1\leq j \leq n}.
$$
The purity weight of the motive is 
$w = w(M) = 2n-2.$ 

\smallskip

Next, we look at $L$-functions. One has the following relations:
$L(s, \sigma) \ = \ L(s, \Psi),$
i.e., the degree-$2n$ $L$-function of $\sigma$ on $\OO(2n)$ corresponding to the inclusion $\OO(2n,\C) \hookrightarrow \GL(2n, \C)$ of dual groups, 
is the standard $L$-function $L(s, \Psi)$ attached to $\Psi.$ Of course, by definition of Tate-twist, one also has 
$L(s, \Psi^T) \ = \ L(s + 1/2, \Psi). $
Another aspect of the dictionary $\Psi^T \leftrightarrow M$ is that 
$L(s, \Psi^T \otimes |\!| \ |\!|^{(1-2n)/2}) \ = \ L(s, M).$ Hence, $L(s, \Psi) = L(s + n-1, M).$ The same relations hold for local and global $L$-functions. One deduces 
for the symmetric square $L$-functions the relation $L(s, \Sym^2, \Psi) = L(s + 2n-2, \Sym^2(M)).$
Since $\Psi$ is a transfer from $\OO(2n)$, it is part of the Arthur classification, that $L(s, \Sym^2, \Psi)$ has a pole at $s = 1$, or that $L(s, \Sym^2(M))$ has a pole at $s = 2n-1.$ 
This happens when the Tate motive $\Q(2-2n)$ appears in $\Sym^2(M)$; in other words, we have an orthogonal structure $\fS$: 
$$
M \otimes M \to \Q(2-2n).
$$

If $g \in G(\fS)$ the group preserving $\fS$, then its multiplier is 
$$
\mu(g) = |\!| g |\!|^{2-2n}. 
$$ 
Also, the multiplier of $\sF_\infty$ on $\Q(2-2n)$ is $(-1)^{2-2n} = 1$ satisfying the parity condition in Thm.\,\ref{thm:gsp-go}. 
The determinant of $M$ is the central character of $\Psi^T \otimes |\!| \ |\!|^{(1-2n)/2}$. The central character of $\Psi$ is trivial since it is a transfer from 
the split orthogonal group, hence the determinant character of $M$ is 
$$
\det(M)(g) = |\!| g |\!|^{(2-2n)n}.
$$
It follows that $\det(M)(g) \mu(g)^{-n} = 1$, i.e., $G(\fS)$ is contained in $\GSO(M)$. Hence, Thm.\,\ref{thm:gsp-go} applies to the motive $M$ as claimed.

\bigskip

\noindent{\it Acknowledgements:} The second author is grateful to the Charles Simonyi Endowment that funded his membership at the Institute for Advanced Study, Princeton, during the Spring and Summer terms of 2018, when the authors got started on this project; he also acknowledges the warm hospitality of the IAS during several subsequent short visits.

\end{document}